\documentclass[a4paper]{amsart}
\usepackage{amsmath,amsthm,amssymb,latexsym,epic,bbm,comment}
\usepackage{graphicx,enumerate,stmaryrd, xcolor}
\usepackage[all,2cell]{xy}
\xyoption{2cell}

\usepackage[active]{srcltx}
\usepackage[parfill]{parskip}
\UseTwocells
\usepackage[
colorlinks=true,
linkcolor=black, 
anchorcolor=black,
citecolor=black, 
urlcolor=black, 
]{hyperref}
\usepackage[e]{esvect}

\newtheorem*{theorem*}{Theorem}
\newtheorem{theorem}{Theorem}
\newtheorem{lemma}[theorem]{Lemma}
\newtheorem{corollary}[theorem]{Corollary}
\newtheorem{proposition}[theorem]{Proposition}

\theoremstyle{definition}

\newtheorem{remark}[theorem]{Remark}

\newtheorem{definition}[theorem]{Definition}

\font\sc=rsfs10
\newcommand{\cC}{\sc\mbox{C}\hspace{1.0pt}}

\font\scc=rsfs7
\newcommand{\ccC}{\scc\mbox{C}\hspace{1.0pt}}

\newcommand{\ucC}{\underline{\cC}}
\newcommand{\uccC}{\underline{\ccC}}

\newcommand{\addcat}{\mathfrak{A}_{\Bbbk}^f}

\newcommand{\Hom}{\mathrm{Hom}}

\newcommand{\End}{\mathrm{End}}

\newcommand{\lmod}{\text{-}\mathrm{mod}}

\newcommand{\proj}{\text{-}\mathrm{proj}}

\newcommand{\add}{\operatorname{add}}
\newcommand{\inj}{\operatorname{inj}}
\newcommand{\Inj}{\operatorname{Inj}}
\newcommand{\comod}{\operatorname{comod}}

\newcommand{\rA}{\mathrm{A}}
\newcommand{\rB}{\mathrm{B}}
\newcommand{\rC}{\mathrm{C}}
\newcommand{\rD}{\mathrm{D}}

\newcommand{\rF}{\mathrm{F}}
\newcommand{\rG}{\mathrm{G}}

\newcommand{\rI}{\mathrm{I}}
\newcommand{\rJ}{\mathrm{J}}
\newcommand{\rK}{\mathrm{K}}

\newcommand{\rM}{\mathrm{M}}
\newcommand{\rN}{\mathrm{N}}

\newcommand{\rQ}{\mathrm{Q}}

\newcommand{\rX}{\mathrm{X}}
\newcommand{\rY}{\mathrm{Y}}
\newcommand{\rZ}{\mathrm{Z}}

\newcommand{\A}{\mathcal{A}}
\newcommand{\B}{\mathcal{B}}
\newcommand{\C}{\mathcal{C}}

\newcommand{\I}{\mathcal{I}}
\newcommand{\J}{\mathcal{J}}
\def\L{{\mathcal{L}}}
\newcommand{\M}{\mathcal{M}}
\newcommand{\N}{\mathcal{N}}
\def\S{\mathcal{S}}

\def\K{{\mathcal{K}}}

\newcommand{\id}{\mathrm{id}}
\newcommand{\one}{\mathbbm{1}}
\newcommand{\Id}{\mathrm{Id}}

\newcommand{\bfM}{\mathbf{M}}\newcommand{\bfK}{\mathbf{K}}
\newcommand{\bfG}{\mathbf{G}}
\newcommand{\bfP}{\mathbf{P}}
\newcommand{\bfN}{\mathbf{N}}

\newcommand{\bfC}{\mathbf{C}}

\newcommand{\ti}{\mathtt{i}}

\begin{document}

\title[Coidempotent subcoalgebras and short exact sequences of $2$-representations]{Coidempotent subcoalgebras and short exact sequences of finitary $2$-representations}
\author{Aaron Chan}
\address{Graduate School of Mathematics, Nagoya University, Furocho, Chikusaku, Nagoya, JAPAN}
\email{aaron.kychan@gmail.com}
\urladdr{}

\author{Vanessa Miemietz}
\address{
School of Mathematics, University of East Anglia, Norwich NR4 7TJ, UK}
\email{v.miemietz@uea.ac.uk}
\urladdr{https://www.uea.ac.uk/~byr09xgu/}


\begin{abstract}
In this article, we study short exact sequences of finitary $2$-representations of a weakly fiat $2$-category. We provide a correspondence between such short exact sequences with fixed middle term and coidempotent subcoalgebras of a coalgebra $1$-morphism defining this middle term. We additionally relate these to recollements of the underlying abelian $2$-representations.
\end{abstract}

\maketitle

\section*{Introduction}

The subject of $2$-representation theory originated from \cite{CR,KhLa,Ro} and is the higher categorical 
analogue of the classical representation theory of algebras. The articles \cite{MM1}--\cite{MM6} develop the $2$-categorical analogue of finite-dimensional algebras and their finite-dimensional modules, by defining and studying finitary $2$-categories and their finitary $2$-representations. One of the fundamental questions in representation theory is to find the simple representations of a given algebra. The question of how to define the $2$-categorical analogue of these was answered in \cite{MM5} where the notion of simple transitive $2$-representations was defined and a Jordan-H\"{o}lder theory for finitary $2$-categories is provided. Since then, there has been considerable effort to classify simple transitive $2$-representations for certain classes of finitary $2$-categories.

Most of the $2$-categories appearing in the categorification of Lie theoretic objects are examples of the so-called weakly fiat $2$-categories. An important defining property of weakly fiat $2$-categories is that, roughly speaking, all $1$-morphisms have adjoints (often called duals for monoidal categories, which, after strictification, can be viewed as $2$-categories with a single object).
The article \cite{MMMT} 
shows that every finitary $2$-representation over a weakly fiat $2$-category can be realised as the category of injective right comodules over a coalgebra $1$-morphism.
This gives a new approach to studying finitary $2$-representations.
It is shown in \cite{MMMZ} that a coalgebra $1$-morphism is cosimple if and only if the corresponding $2$-representation is simple transitive.
In other words, classifying simple transitive $2$-representations is equivalent to classifying cosimple coalgebra $1$-morphisms (up to Morita--Takeuchi equivalences).

This article takes a slightly different direction.  After all, another important aspect of the theory of modules over algebras is homological algebra, i.e. how to build all representations from simple ones.  The $2$-analogue for homological theory associated to finitary $2$-categories has so far only been studied in \cite{CM}, where an analogue of Ext-groups are introduced and studied.  In this article, instead, we look back at the definition of short exact sequence of (finitary) $2$-representations used in \cite{CM} (originally from \cite{SVV}), and relate them to comodules categories over coalgebra $1$-morphisms.
The questions we ask are the following.
\begin{itemize}
\item How do we realise a finitary sub-$2$-representation in the language of comodule theory over coalgebra $1$-morphisms?
\item When can we fit the quotient morphism of $2$-representations induced by a subcoalgebra into a short exact sequence of $2$-representation?
\item What is the relation between the coalgebra $1$-morphisms generating the three finitary $2$-representations appearing in a short exact sequence of $2$-representations?
\end{itemize}

It turns out that the answer is closely related to coidempotent subcoalgebras (see Definition \ref{def:coidem}) and recollements of abelian categories.  More precisely, our main theorem (Theorem \ref{thm-main}) states that
\begin{itemize}
\item given a coidempotent subcoalgebra $\rD$ of a coalgebra $1$-morphism $\rC$, we can construct a coalgebra $1$-morphism $\rA$ from a certain injective $\rC$-comodule $\rI$ such that there is a short exact sequence
  $$0 \longrightarrow \inj_{\uccC}(\rA) \xrightarrow{-\square_\rA\rI} \inj_{\uccC}(\rC)\xrightarrow{-\square_\rC\rD}\inj_{\uccC}(\rD) \longrightarrow 0$$
  of $2$-representations, where $-\square_\rY\rX$ denotes the cotensor product functor;
\item given a short exact sequence of $2$-representations
  $$0\longrightarrow\bfN\longrightarrow\bfM\longrightarrow\bfK\longrightarrow 0$$
  and choosing a coalgebra $1$-morphism $\rC$ with $\bfM\cong \inj_{\uccC}(\rC)$, there exists a subcoalgebra $\rD$ of $\rC$, unique up to isomorphism and necessarily coidempotent, such that $\inj_{\uccC}(\rD)$ is equivalent to the quotient $2$-representation $\bfK$.
\end{itemize}

Moreover, passing to the abelianised $2$-representations, in the above situation, we have a recollement of abelian categories
\begin{equation*}
\xymatrix{  \comod_{\uccC}(\rD)\ar^{-\square_\rD\rD_\rC}[rr]&& \comod_{\uccC}(\rC)\ar^{[\rI,-]}[rr]\ar@<1ex>@/^/^{-\square_\rC\rD}[ll]\ar@<-2ex>@/_/_{[\rD,-]}[ll]&& \comod_{\uccC}(\rA),\ar@<1ex>@/^/^{-\square_\rA\rI}[ll]\ar@<-2ex>@/_/_{[[\rI,\rC],-]}[ll]
}\end{equation*}
where $[\rX,-]$ denotes the internal hom functor.

The paper is organised as follows. In Section \ref{oldstuff}, we provide a summary of the setup and results from previous articles on the subject which we need for our purposes. In Section \ref{prelims}, we discuss some preliminary results about recollements and functors between comodules categories. We also provide a correspondence between subcoalgebras of a given coalgebra and subcategories of its comodule categories that are closed under subobjects, quotients and closed under the action by the $2$-category, generalising results in \cite{NT}. In Section \ref{coidext}, we define coidempotent subcoalgebras, show that they correspond to Serre subcategories of the category of comodules, and discuss their relationship with recollements. This then leads to the statement and the proof of the main theorem in the final subsection. Finally, we provide some examples in Section \ref{sec:eg}.

{\bf Acknowledgements.} AC is supported by a JSPS International Research Fellowship. Part of this research was carried out during a visit of AC to the University of East Anglia, whose hospitality is gratefully acknowledged.

\section{Recollections.}\label{oldstuff}

Let $\Bbbk$ be an algebraically closed field.

\subsection{$2$-categories and $2$-representations}\label{s2.1}

We start by recalling some terminology on finitary categories and $2$-categories. We refer to reader to  \cite{Le,McL} for more detail on general $2$-categories and to \cite{MM1,MM2,MM3,MM4,MM5,MM6} for more detail on 
$2$-representations of finitary $2$-categories. 

A $\Bbbk$-linear category is called {\bf finitary} if it is idempotent complete, has only finitely many isomorphism classes of indecomposable objects and all morphism spaces are finite dimensional. The collection of finitary $\Bbbk$-linear categories, together with additive $\Bbbk$-linear functors and all natural transformations between such functors, forms a $2$-category denoted by $\addcat$.

In \cite{MM1}, a {\bf finitary} $2$-category $\cC$ was defined to be a $2$-category such that
\begin{itemize}
\item $\cC$ has finitely many objects;
\item each morphism category $\cC(\mathtt{i},\mathtt{j})$ is in $\addcat$;
\item horizontal composition is biadditive and bilinear;
\item for each $\mathtt{i}\in\cC$, the identity $1$-morphism $\mathbbm{1}_\mathtt{i}$ is indecomposable.
\end{itemize}
We denote by $\circ_{0}$ and $\circ_{1}$ the horizontal and vertical compositions in $\cC$, respectively.

A finitary $2$-category $\cC$ is called {\bf weakly fiat} if it has a weak
anti-equivalence $(-)^*$ reversing the direction of both $1$- and $2$-morphisms, such that, for a
$1$-morphism $\mathrm{F}$, the pair $(\mathrm{F},\mathrm{F}^*)$ is an adjoint pair, see \cite[Subsection~2.4]{MM1}. It is called {\bf fiat} if $(-)^*$ is weakly involutive. We denote the weak inverse of $(-)^*$ by ${}^*(-)$, obtaining another adjoint pair $({}^*\mathrm{F},\mathrm{F})$.

A {\bf finitary} $2$-representation of $\cC$ is a $2$-functor from $\cC$ to $\addcat$.  An important example of a finitary $2$-representation is, for each $\ti\in\cC$, the {\bf principal}
$2$-representation $\mathbf{P}_{\mathtt{i}}:=\cC(\mathtt{i},{}_-)$. 

We can (injectively) abelianise both the $2$-category $\cC$ and, for a $2$-representation $\bfM$, the category $\M:=\prod_{\ti\in\ccC}\bfM(\ti)$ and use the notation $\underline{(-)}$ for the injective abelianisation ($2$)-functor. For the $2$-category, this needs to be done in a rather technical way, see \cite[Section 3.2]{MMMT} to preserve strictness of horizontal composition. Note that, provided that $\cC$ is weakly fiat, composition in $\ucC$ is left exact in both variables. Indeed, left and right multiplication by $1$-morphisms in $\cC$ is exact thanks to the existence of adjoints, and all $1$-morphisms of $\ucC$ can be regarded as kernels of $2$-morphisms in $\cC$, whence application of the snake lemma yields the claim, cf. \cite[Subsection 3.1]{MMMZ}.

For $\M$, it is equivalently possible to use the classical diagrammatic abelianisation, see \cite{Fr}, or \cite[Section 3.1]{MMMT} for a presentation adapted to our notation. This induces an {\bf abelian} $2$-representation $\underline{\bfM}$ on $\underline{\M}$.

Both finitary and abelian $2$-representations of $\cC$ form $2$-categories,
denoted $\cC$-afmod and $\cC$-mod, respectively, in which $1$-morphisms are strong 
$2$-natural transformations, which we also simply call morphisms of $2$-represenations, and $2$-morphisms are modifications, see \cite[Section 2]{MM3} for details. 

In slight abuse of notation, we will, for any $2$-representation $\mathbf{M}$, write $\rF\,X$ rather than $\bfM(\rF)(X)$.

A $2$-representation $\mathbf{M}\in\cC$-afmod is said to be {\bf transitive}, cf. \cite[Subsection 3.1]{MM5},
if, for any indecomposable objects $X,Y\in\M$, there exists a $1$-morphism $\rF$ in $\cC$ such that 
$Y$ is isomorphic to a direct summand of $\rF\,X$. We say that a transitive $2$-representation $\bfM$
is {\bf simple transitive}, cf. \cite[Subsection 3.5]{MM5}, if
$\M$ has no proper $\cC$-invariant ideals.
In \cite[Section 4]{MM5}, it was proved that every $\bfM\in\cC$-afmod has a {\bf weak Jordan-H{\"o}lder series} with transitive subquotients, and the list of their respective simple transitive quotients is unique up to permutation and equivalence.

\subsection{Coalgebra $1$-morphisms and their comodule categories.}

A {\bf coalgebra $1$-morphism} in $\ucC$ is a coalgebra object in $\coprod_{\ti\in\ccC}\ucC(\ti,\ti)$, i.e. a direct sum $\rC$ of $1$-morphisms in $ \coprod_{\ti\in\ccC}\ucC(\ti,\ti)$equipped with $2$-morphisms $\mu_\rC\colon \rC\to\rC\rC$ and $\varepsilon_\rC\colon\rC\to\mathbbm{1}=\bigoplus_{\ti\in\ccC}\mathbbm{1}_\ti$, called comultiplication and counit respectively, satisfying coassociativity $(\mu_\rC\circ_0\id_\rC)\circ_1\mu_\rC = (\id_\rC\circ_0\mu_\rC)\circ_1\mu_\rC$ and counitality $(\id_\rC\circ_0\varepsilon_\rC)\circ_1\mu_\rC = \id_\rC = (\varepsilon_\rC\circ_0\id_\rC)\circ_1\mu_\rC$.

A {\bf right} (resp. {\bf left}) {\bf comodule} over $\rC$ is a $1$-morphism $\rM$ in $\ucC$ together with a coaction $\rho_\rM\colon \rM\to\rM\rC$ (resp. $\lambda_\rC\colon \rM\to\rC\rM$) such that $(\id_\rM\circ_0\mu_\rC)\circ_1\rho_\rM = (\rho_\rM\circ_0\id_\rC)\circ_1\rho_\rM$ and $(\id_\rM\circ_0\varepsilon_\rC)\circ_1\rho_\rM = \id_\rM$ (resp. $(\mu_\rC\circ_0\id_\rM)\circ_1\rho_\rM = (\id_\rC\circ_0\lambda_\rM)\circ_1\rho_\rM$ and $(\varepsilon_\rC\circ_0\id_\rM)\circ_1\rho = \id_\rM$). Note that the last condition implies that all coaction maps are monomorphisms in $\ucC$.

The {\bf cotensor product over $\rC$} of a right $\rC$-comodule $\rM$ with a left $\rC$-comodule $\rN$ is the kernel of the map $$\rM\rN \xrightarrow{\rho_\rM\circ_0\id_\rN-\id_\rM\circ_0\lambda_\rN}\rM\rC\rN.$$

\subsection{Internal homs and $2$-representations.}\label{coalgrecall} Let $\cC$ be a weakly fiat $2$-category.
This subsection is essentially a summary of  \cite[Sections 4]{MMMT}. Note that results there were stated for a fiat $2$-category, but none of the proofs use involutivity of $(-)^*$, hence all proofs go through verbatim for the weakly fiat case.

Let $\mathbf{M}$ be a finitary $2$-representation of $\cC$ and $N\in {\M}$. Recall the internal hom functor $[N,-]\colon \underline{\M} \to \ucC$, which is defined as the left adjoint to the evaluation of the action on $N$, i.e.
$$\Hom_{\underline{\M}}(-, \rF N) \cong \Hom_{\underline{\ccC}}([N,-], \rF) $$ for all $\rF \in \ucC$. The internal hom $[N,N]$ has the structure of a coalgebra $1$-morphism and for any $M\in \underline{\M}$, $[N,M]$ has the structure of a right $[N,N]$-comodule in $\underline{\cC}$. The category consisting of such right $[N,N]$-comodule in $\underline{\cC}$ carries the structure of an abelian $2$-representation of $\cC$, denoted by $\comod_{\uccC}[N,N]$, and the finitary  $2$-representation on the subcategory of injective right $[N,N]$-comodules is denoted by $\inj_{\uccC}[N,N]$. The latter is equivalent to the additive closure in $\comod_{\uccC}[N,N]$ of $\{\rF[N,N] \mid \rF\in \cC\}$.
Note that, as shown in \cite[Proof of Lemma 6]{MMMT}, $\rF[X,Y]\cong [X,\rF Y]$ for all $X, Y\in \M$ and all $\rF\in\cC$; the same holds also for $X, Y\in \underline{\M}$ by the same proof.
Note also that $\comod_{\uccC}[N,N]$ is equivalent to $\underline{\inj_{\uccC}[N,N]}$. 

In \cite[Section 4]{MMMT}, it was shown that when $\bfM$ is transitive, then the realisation morphism $[N,-]$ defines an equivalence of abelian $2$-representations between $\underline{\bfM}$ and $\comod_{\uccC}[N,N]$, and also restricts to an equivalence of finitary $2$-representations between $\bfM$ and $\inj_{\uccC}[N,N]$.
In fact, the same proof works for any $\bfG_\bfM(N)$, i.e. 
for arbitrary $N\in \M$, the realisation morphism induces an equivalence of finitary 2-representations between $\bfG_\bfM(N)$ and $\inj_{\uccC}[N,N]$.  In particular, one can always realise a finitary 2-representation as $\inj_{\uccC}[N,N]$ by taking $N$ as the direct sum of all indecomposable objects (up to isomorphisms).
As such, from now on, we do not distinguish between comodules (resp. injective comodules) over a coalgebra 1-morphism and objects of an abstract abelian (respectively, finitary) $2$-representation.

\subsection{Extensions of $2$-representations.}

A sequence
\begin{equation}\label{eq1}
\xymatrix{
0\ar[r]& \mathcal{A}\ar[r]^{\mathrm{F}}& \mathcal{B}\ar[r]^{\mathrm{G}}& \mathcal{C}\ar[r]& 0
} 
\end{equation}
in $\addcat$ will be called {\bf short exact} (cf. \cite[Subsection~2.2.1]{SVV}) 
provided that
\begin{itemize}
\item $\mathrm{F}$ is full and faithful;
\item $\mathrm{G}$ is full and dense;
\item the kernel of $\mathrm{G}$ coincides with the ideal of
$\mathcal{B}$ generated by $\mathrm{F}(\mathcal{A})$.
\end{itemize}

A sequence of morphisms $\Phi,\Psi$ of additive $2$-representations
\begin{equation*}
\xymatrix{
0\ar[r]& \bfN\ar[r]^{\Phi}& \bfM\ar[r]^{\Psi}& \bfK \ar[r]& 0
} 
\end{equation*}
will be called an {\bf extension} of $2$-representations, provided that the underlying sequence
\begin{equation*}
\xymatrix{
0\ar[r]& \N\ar[r]^{\Phi}& \M\ar[r]^{\Psi}& \K \ar[r]& 0
} 
\end{equation*}
is short exact in $\addcat$, where in the second sequence $\Phi$ and $\Psi$ refer to the underlying additive functors.

\section{Preliminary results}\label{prelims}

In this section, we collect some preliminary results leading towards our main theorem.

\subsection{Recollements of abelian categories.}

Recall that a diagram 
\begin{equation*}
\xymatrix{  \A\ar^{i}[rr]&& \B\ar^{e}[rr]\ar@<1ex>@/^/^{p}[ll]\ar@<-2ex>@/_/_{q}[ll]&& \C\ar@<1ex>@/^/^{r}[ll]\ar@<-2ex>@/_/_{l}[ll]
}\end{equation*}
of abelian categories is a recollement provided that
\begin{itemize}
\item $(q,i,p)$ and $(l,e,r)$ are adjoint triples;
\item the functors $l,r$ and $i$ are fully faithful;
\item the image of $i$ is a Serre subcategory, which is the kernel of $e$. 
\end{itemize}

\begin{lemma}\label{recolllem}
Let 
\begin{equation*}
\xymatrix{  \A\ar^{i}[rr]&& \B\ar^{e}[rr]\ar@<1ex>@/^/^{p}[ll]\ar@<-2ex>@/_/_{q}[ll]&& \C\ar@<1ex>@/^/^{r}[ll]\ar@<-2ex>@/_/_{l}[ll]
}\end{equation*}
be a recollement of abelian categories with enough injectives, where $(q,i,p)$ and $(l,e,r)$ are adjoint triples. Then the sequence given by $r$ and $p$ restricts to a short exact sequence of additive categories
\begin{equation}\label{injrecoll}
0\to \Inj\C\overset{r}{\to}\Inj\B\overset{p}{\to}\Inj\A \to 0.
\end{equation}
between the full subcategories of injective objects.
\end{lemma}

\proof
The sequence restricts since both $r$ and $p$ are right adjoints to exact functors and hence preserve injectives. 

By the definition of recollement, $r$ is fully faithful.  Since $pi$ is naturally isomorphic to the identity functor on $\A$ (see e.g.\cite[Proposition 2.7(ii)]{PV}), $p$ is necessarily full and dense.


It remains to show that the kernel of $p$ coincides with the ideal $\I$ in $\Inj\B$ generated by (the full subcategory given by) the essential image of $r$ restricted to $\Inj\C$.
It is well-known that $pr=0$ (see e.g.\cite[Proposition 2.7(ii)]{PV}), so it immediately follows that, considering the restricted sequence \eqref{injrecoll}, $\I$ is contained in the kernel of $p$.  For simplicity, we say that an object is in $\I$ if its identity morphism is in $\I$.

Assume that $Q_1,Q_2\in \Inj\B$ are both not annihilated by $p$, and hence are not objects in $\I$. We claim that if $f\colon Q_1 \to Q_2$ is annihilated by $p$, it factors over some $I\in \I$.
Indeed, as $ip(M)$ is the maximal subobject of $M$ with composition factors belonging to $i(\A)$ for any $M\in\B$, $ip$ is a subfunctor of $\Id_{\B}$.  Thus, we have a commutative diagram of solid arrows
$$\xymatrix{
ip(Q_1)\ar@{^{(}->}[d]\ar^{ip(f)=0}[rr]&&ip(Q_2)\ar@{^{(}->}[d]\\
Q_1\ar^{f}[rr]\ar@{->>}[d]&&Q_2\\
Q_1/ip(Q_1)\ar@{-->}^{\bar f}[urr]&&
}$$
meaning that $f$ factors over $Q_1/ip(Q_1)$ as indicated by the dashed arrow $\bar f$.

Considering the exact sequence
$$0\to ip(Q_1)\to Q_1\to re(Q_1)$$
(cf. \cite[Proposition 4.2]{FP}, \cite[Proposition 2.6(ii)]{Ps})
and letting $I'$ be the injective hull of $e(Q_1)\in \C$, we obtain a monomorphism $Q_1/ip(Q_1)\hookrightarrow r(I')$, and hence
the injective hull $I$ of $Q_1/ip(Q_1)$, which is a direct summand of $r(I')$, is in $\I$.
By injectivity of $Q_2$, $\bar f$ now factors over $I$, so $f$ factors over $I\in\I$, as claimed.
\endproof

\subsection{Functors between comodule categories.}

From now on, $\cC$ will denote a weakly fiat $2$-category. 

\begin{lemma}\label{adjlemma}
Let $\rC, \rC'$ be coalgebra $1$-morphisms in $\underline{\cC}$  and  $Y$ a $\rC, \rC'$-bicomodule.
\begin{enumerate}[$($i$)$]
\item\label{adjlemma1} For any $M\in \comod_{\uccC}(\rC')$, the internal hom $[Y,M]$ is a right $\rC$-comodule in $\ucC$.
\item\label{adjlemma2} $[Y,-]\colon \comod_{\uccC}(\rC') \to \comod_{\uccC}(\rC)$ is left adjoint to $-\square_\rC Y$.
\end{enumerate}
\end{lemma}

\proof
Both statements are proved in exactly the same way as in the classical case of coalgebras over a field, see \cite[12.6, 12.7]{BW}. 
\endproof

\begin{lemma}\label{12.8} Let $\rC, \rC'$ be coalgebra $1$-morphisms in $\underline{\cC}$  and  $\rY$ a $\rC, \rC'$-bicomodule. The following statements are equivalent:
\begin{enumerate}[$($a$)$]
\item $\rY\in \inj_{\uccC}(\rC')$.
\item $[\rY,-]$ is exact.
\end{enumerate}
If either condition is satisfied, we have $[\rY,-]\cong -\square_{\rC'} [\rY,\rC']$ as functors from $\comod_{\uccC}(\rC')$ to $\comod_{\uccC}(\rC)$
\end{lemma}

\proof
The same proof as in \cite[12.8, 23.7]{BW} shows that $[\rY,-]:\comod_{\uccC}(\rC')\to \comod_{\uccC}(\rC)$ is exact if and only if $\rI\square_{\rC}\rY$ is injective for all injective $\rC$-comodules $\rI$.
In our setting, since every injective $\rI$ is direct summand in $\comod_{\uccC}(\rC)$ of $\rF\rC$ for some $1$-morphism $\rF$, this is equivalent to $\rC\square_{\rC}\rY\in \inj_{\uccC}(\rC')$, but $\rC\square_\rC \rY\cong \rY$.
The last statement is proved in the same way as in loc. cit.
\endproof

Since $[\rC,\rC]\cong \rC$ by the definition of realisation morphism, an immediate consequence of Lemma \ref{12.8} is the following result.

\begin{corollary}\label{idcohom}
For a coalgebra $1$-morphism $\rC$ in $\underline{\cC}$, we have an isomorphism between $[\rC,-]$ and the identity functor on $\comod_{\uccC}(\rC)$.
\end{corollary}

\begin{lemma}\label{Iinfromleft}
Let $\rC$ be  a coalgebra $1$-morphism in $\underline{\cC}$, $\rI\in \inj_{\uccC}(\rC)$. Then 
$$\rI\square_\rC[\rM,\rC] \cong [\rM,\rI]$$
for all $\rM\in \comod_{\uccC}(\rC)$. 
\end{lemma}

\proof
Any $\rI\in \inj_{\uccC}(\rC)$ is a direct summand in $\comod_{\uccC}(\rC)$ of $\rF\rC$ for some $1$-morphism $\rF\in\cC$. Since all functors are additive, the claim follows from
$$\rF\rC\square_\rC[\rM,\rC] \cong \rF[\rM,\rC]\cong [\rM,\rF\rC].$$
\endproof

\subsection{The comodule category of a subcoalgebra and related functors}
Let $\rC=(\rC,\mu_\rC,\epsilon_\rC)$ be a coalgebra $1$-morphism.
By a {\bf subcoalgebra} $\rD$ of $\rC$, we mean a coalgebra $1$-morphism $\rD=(\rD,\mu_\rD,\epsilon_\rD)$ together with a monomorphism $\iota:\rD\hookrightarrow\rC$ in $\underline{\cC}$ satisfying $\mu_\rD\circ_1\iota=(\iota\circ_0\iota)\circ_1\mu_\rD$ and $\epsilon_\rD=\epsilon_\rC\circ_1\iota$.

Note that for any right $\rD$-comodule $\rN$ with coaction map $\rho_\rN^\rD:\rN\to\rN\rD$, one naturally obtains a right $\rC$-comodule by post-composing $\rho_\rN^\rD$ with $\id_\rN\circ_0\iota$.
This construction give rise to a functor $-\square_\rD\rD_\rC$ (see the lemma below and \cite[Section 3.4]{MMMZ}).
In particular, a right $\rC$-comodule is in the essential image of $-\square_\rD\rD_\rC$ if its coaction map $\rho_\rM^\rC:\rM\to \rM\rC$ factors through $\id_\rM\circ_0\iota$.
This fact will be used throughout the rest of the article. 
 
\begin{lemma}\label{lemclosed}
Let $\rC$ be a coalgebra $1$-morphism in $\underline{\cC}$ and $\rD \overset{\iota}{\hookrightarrow} \rC$ be a subcoalgebra with cokernel $\rC \overset{\pi}{\twoheadrightarrow} \rJ$. The natural morphism of $2$-representations $\comod_{\uccC}(\rD) \to\comod_{\uccC}(\rC)$ given by $-\square_\rD \rD_\rC$ is fully faithful, exact, and the subcategory it defines is closed under quotients and subobjects.
\end{lemma}

\proof
The fact that  $-\square_\rD \rD_\rC$ is  faithful is obvious from the definition. 
By injectivity of ${}_\rD\rD$, it follows from Lemma \ref{12.8} that $[\rD,-]:\comod_{\uccC}(\rD)\to \comod_{\uccC}(\rC)$ is exact and  $[\rD,-]\cong -\square_\rD [\rD,\rD]_\rC \cong -\square_\rD \rD_\rC$, hence $-\square_\rD \rD_\rC$ is exact.

To see that it is full consider a morphism $f\colon \rM\to \rN$ between two objects isomorphic to $\rM'\square_\rD \rD_\rC$ and  $\rN'\square_\rD \rD_\rC$ respectively, i.e. both coaction $\rho_\rM$ and $\rho_\rN$ factor over $\rho_\rM^\rD\colon \rM\to \rM\rD$ and $\rho_\rN^\rD\colon \rN\to \rN\rD$ respectively. Consider the diagram
$$\xymatrix@C=40pt{
\rM\ar^{f}[rrr]\ar^{\rho_\rM^\rD}[rd]\ar_{\rho_\rM}[dd]&&&\rN\ar[dd]^{\rho_\rN}\ar_{\rho_\rN^\rD}[ld]\\
&\rM\rD\ar^{f\circ_0\id_\rD}[r]\ar^{\id_\rM\circ_0\iota}[dl]&\rN\rD\ar^{\id_\rN\circ_0\iota}[dr]&\\
\rM\rC\ar_{f\circ_0\id_\rC}[rrr]&&&\rN\rC.
}$$
where the triangles, the outer square and the lower trapezium commute. Then 
\begin{equation*}
\begin{split}
(\id_\rN\circ_0\iota)\circ_1\rho_\rN^\rD\circ_1 f&= \rho_\rN\circ_1 f = (f\circ_0\id_\rC)\circ_1\rho_\rM\\
&=(f\circ_0\id_\rC)\circ_1(\id_\rM\circ_0\iota)\circ_1\rho_\rM^\rD\\
&= (\id_\rN\circ_0\iota)\circ_1(f\circ_0\id_\rD)\circ_1\rho_\rM^\rD.
\end{split}
\end{equation*}
Since $(\id_\rN\circ_0\iota)$ is mono, $\rho_\rN^\rD\circ_1 f=(f\circ_0\id_\rD)\circ_1\rho_\rM^\rD$, so $f$ is induced from a morphism in $\comod_{\uccC}(\rD)$ and $-\square_\rD \rD_\rC$ is full.

Let $\rM$ be isomorphic to an object of the form $\rM'\square_\rD \rD_\rC$, i.e. the coaction $\rho_\rM\colon \rM\to \rM\rC$ factors over the inclusion $\id_\rM\circ_0\iota\colon \rM\rD\hookrightarrow \rM\rC$. 

To show closure under quotients,
let $f\colon \rM\twoheadrightarrow \rN$ be an epimorphism in $\comod_{\uccC}(\rC)$. Consider the solid part of the diagram
$$\xymatrix@C=40pt{
\rM\ar@{->>}^{f}[rrr]\ar_{\rho_\rM^\rD}[rd]\ar_{\rho_\rM}[dd]&&&\rN\ar^{\rho_\rN}[dd]\ar@{-->}[dl]_{\sigma}\\
&\rM\rD\ar^{f\circ_0\id_\rD}[r]\ar^{\id_\rM\circ_0\iota}[dl]&\rN\rD\ar^{\id_\rN\circ_0\iota}[dr]&\\
\rM\rC\ar^{f\circ_0\id_\rC}[rrr]\ar_{\id_\rM\circ_0\pi}[d]&&&\rN\rC\ar_{\id_\rN\circ_0\pi}[d]\\
\rM\rJ\ar^{f\circ_0\id_\rJ}[rrr]&&&\rN\rJ.
}$$
Since $(\id_\rM\circ_0\pi)\circ_1\rho_\rM = (\id_\rM\circ_0\pi)\circ_1(\id_\rM\circ_0\iota)\circ_1\rho_\rM^\rD = (\id_\rM\circ_0(\pi\circ_1\iota))\circ_1\rho_\rM^\rD=0$, we have 
$ (\id_\rN\circ_0\pi)\circ_1\rho_\rN\circ_1 f =(f\circ_0\id_\rJ)\circ_1(\id_\rM\circ_0\pi)\circ_1\rho_\rM  = 0$ and, since $f$ is epi, $ (\id_\rN\circ_0\pi)\circ_1\rho_\rN=0$. Hence $\rho_\rN$ factors over the kernel of $\id_\rN\circ_0\pi$, which, by left exactness of horizontal composition with $\id_\rN$ is $\id_\rN\circ_0\iota$. This yields the dashed arrow $\sigma$. Now we have
\begin{equation*}
\begin{split}
(\id_\rN\circ_0\iota)\circ_1\sigma\circ_1 f &= \rho_\rN\circ_1 f =(f\circ_0\id_\rC)\circ_1\rho_\rM\\
&=(f\circ_0\id_\rC)\circ_1(\id_\rM\circ_0\iota)\circ_1\rho_\rM^\rD \\
&= (\id_\rN\circ_0\iota)\circ_1(f\circ_0\id_\rD)\circ_1\rho_\rM^\rD.
\end{split}
\end{equation*}
As $(\id_\rN\circ_0\iota)$ is mono, it follows that $\sigma\circ_1 f =(f\circ_0\id_\rD)\circ_1\rho_\rM^\rD$, so the coaction on $\rN$ indeed factors over $\rN\rD$ as claimed.

To show closure under subobjects, let $f\colon \rN\hookrightarrow \rM$ be a monomorphism.
Consider the solid part of the diagram
$$\xymatrix@C=40pt{
\rN\ar@{^{(}->}^{f}[rrr]\ar@{-->}^{\sigma}[rd]\ar_{\rho_\rN}[dd]&&&\rM\ar[dd]^{\rho_\rM}\ar_{\rho_\rM^\rD}[ld]\\
&\rN\rD\ar^{f\circ_0\id_\rD}[r]\ar^{\id_\rN\circ_0\iota}[dl]&\rM\rD\ar^{\id_\rM\circ_0\iota}[dr]&\\
\rN\rC\ar^{f\circ_0\id_\rC}[rrr]\ar_{\id_\rN\circ_0\pi}[d]&&&\rM\rC\ar_{\id_\rM\circ_0\pi}[d]\\
\rN\rJ\ar^{f\circ_0\id_\rJ}[rrr]&&&\rM\rJ.
}$$
As before, $(\id_\rM\circ_0\pi)\circ_1\rho_\rM =0$, so $(\id_\rM\circ_0\pi)\circ_1\rho_\rM \circ_1 f=(f\circ_0\id_\rJ)\circ_1(\id_\rN\circ_0\pi)\circ_1\rho_\rN  = 0$ and since $f\circ_0\id_\rJ$ is a monomorphism (using left exactness of horizontal composition with $\id_\rJ$), furthermore, $(\id_\rN\circ_0\pi)\circ_1\rho_\rN  = 0$. Hence, as above, $\rho_\rN $ factors over $\rN\rD$, giving the dashed arrow $\sigma$.
Similarly to before,
\begin{equation*}
\begin{split}
(\id_\rM\circ_0\iota)\circ_1(f\circ_0\id_\rD)\circ_1\sigma&= 
(f\circ_0\id_\rC)\circ_1(\id_\rN\circ_0\iota)\circ_1\sigma\\
&=(f\circ_0\id_\rC)\circ_1\rho_\rN=\rho_\rM\circ_1 f\\
&=(\id_\rM\circ_0\iota)\circ_1\rho_\rM^\rD\circ_1 f\\
\end{split}
\end{equation*}
and thanks to monicity of $\id_\rM\circ_0\iota$, we conclude $(f\circ_0\id_\rD)\circ_1\sigma = \rho_\rM^\rD\circ_1 f$.

It is immediate that in both cases that $\sigma$ defines a right coaction on $\rN$.
Indeed, in general, if a right $\rC$-coaction $\rho_\rN\colon \rN\to \rN\rC$ factors over the inclusion $\id_\rN\circ_0\iota\colon \rN\rD\to\rN\rC$ via a map $\sigma$, we have  
\begin{equation*}\begin{split}
(\id_\rN\circ_0\iota\circ_0\iota)\circ_1(\sigma\circ_0\id_\rD)\circ_1\sigma 
&= [((\id_\rN\circ_0\iota)\circ_1 \sigma)\circ_0 \id_\rC]\circ_1 (\id_\rN\circ_0\iota)\circ_1\sigma\\
&= (\rho_\rN\circ_0\id_\rC) \circ_1\rho_\rN\\
&= (\id_\rN\circ_0\mu_\rC) \circ_1\rho_\rN\\
&=(\id_\rN\circ_0\mu_\rC) \circ_1 (\id_\rN\circ_0\iota)\circ_1\sigma\\
&=(\id_\rN\circ_0\iota\circ_0\iota)\circ_1(\id_\rN\circ_0\mu_\rD)\circ_1\sigma
\end{split}\end{equation*}
where the first equality uses the interchange law twice, the second and fourth equalities are the definition of $\sigma$ , the third equality comes from $\rho_N$ being a coaction, and the last equality from $\iota$ being a coalgebra map. Cancelling the monomorphism $\id_\rN\circ_0\iota\circ_0\iota$ implies the first comodule axiom. For the second, we compute
$$(\id_\rN\circ_0\varepsilon_\rD)\circ_1\sigma =  (\id_\rN\circ_0\varepsilon_\rC)\circ_1(\id_\rN\circ_0\iota)\circ_1\sigma (\id_\rN\circ_0\varepsilon_\rC)\circ_1\rho_\rN=\id_\rN.$$
\endproof

\begin{lemma}\label{MMMZCor7Lem8}
Let $\rC$ be a coalgebra $1$-morphism in in $\underline{\cC}$, and $\rD \overset{\iota}{\hookrightarrow} \rC$ a subcoalgebra.
\begin{enumerate}[$($i$)$]
\item\label{MMMZ7} $-\square_\rD\rD\square_\rC\rD$ is naturally isomorphic to the identity morphism on $\comod_{\uccC}(\rD)$.
\item\label{MMMZ8} There is a monic natural transformation from $-\square_{\rC}\rD\square_{\rD}\rD_\rC$ to the identity morphism on $\comod_{\uccC}(\rC)$.
\end{enumerate}
\end{lemma}
\proof
Denote by $\Psi$ the morphism $(-\square_\rD \rD_\rC):\comod_{\uccC}(\rD)\to\comod_{\uccC}(\rC)$ and by $\Phi$ the morphism $-\square_\rC\rD:\comod_{\uccC}(\rC)\to\comod_{\uccC}(\rD)$.
Note that $\Psi\cong [\rD,-]$ as argued in the proof of Lemma \ref{lemclosed}, and $(\Psi\cong [\rD,-],\Phi=-\square_{\rC}\rD)$ is an adjoint pair by Lemma \ref{adjlemma}.
Now (i) and (ii) are exactly the same as \cite[Corollary 7]{MMMZ} and \cite[Lemma 8]{MMMZ} respectively.
\endproof

\begin{lemma}\label{realise-subcat}
Suppose $\S$ is a full subcategory of $\comod_{\uccC}(\rC)$ that is $\cC$-stable, subobject-closed, and quotient-closed.
Let $i$ be the (fully faithful exact) embedding of $\S$ into $\comod_{\uccC}(\rC)$ and $p$ be its right adjoint.
Then $\rD:=[ip(\rC),ip(\rC)]$ is a subcoalgebra of $\rC$ so that $-\square_{\rD}\rD_\rC:\comod_{\uccC}(\rD)\to\comod_{\uccC}(\rC)$ induces an equivalence between $\comod_{\uccC}(\rD)$ and $\S$.
\end{lemma}
\begin{proof}
Consider the right $\rC$-comodule $\rB$ given by the sum of all images of right $\rC$-comodule morphisms of the form $f:\rM\to \rC$ with $\rM\in \S$.
Since $\S$ is quotient-closed, we have $\rB\in \S$.
In particular, $\rB$ coincides with $ip(\rC)$ (which is the sum of all subobjects of $\rC$ in $\S$), and the counit of the adjoint pair $(i,p)$ therefore defines a monomorphism $\iota':\rB\to \rC$.

For any $\rF\in \cC$, there is an exact sequence
\[
0\to \Hom_{\comod_{\uccC}(\rC)}(\rF\rB,\rB)\xrightarrow{\Hom(\rF\rB,\iota')=\iota'\circ-} \Hom_{\comod_{\uccC}(\rC)}(\rF\rB,\rC)
\]
in $\ucC$.
Since $\S$ is $\cC$-stable, we have $\rF\rB\in\S$.
By the construction of $\rB$, every morphism from an object of $\S$ to $\rC$ factors through $\iota'$, so the morphism in the above exact sequence is surjective and hence an isomorphism.
Using the adjoint pairs $([\rB,-],-\cdot\rB)$ and $(\rF,\rF^*)$, and the fact that $\rF[X,Y]\cong [X,\rF Y]$ for all $X,Y\in \comod_{\uccC}(\rC)$ and all $1$-morphism $\rF$, we obtain the following commutative diagram
\[
\xymatrix@C=60pt{
\Hom_{\comod_{\uccC}(\rC)}(\rF\rB,\rB) \ar[r]^{\iota'\circ-}\ar[d]^{\sim} & \Hom_{\comod_{\uccC}(\rC)}(\rF\rB,\rC)\ar[d]^{\sim}\\
\Hom_{\uccC}([\rB,\rF\rB],\mathbbm{1}) \ar[r]^{-\circ [\iota', \rF\rB]}\ar[d]^{\sim} & \Hom_{\uccC}([\rC,\rF\rB],\mathbbm{1}) \ar[d]^{\sim}\\
\Hom_{\uccC}(\rF[\rB,\rB],\mathbbm{1}) \ar[r]^{-\circ \rF[\iota', \rB]}\ar[d]^{\sim} & \Hom_{\uccC}(\rF[\rC,\rB],\mathbbm{1})\ar[d]^{\sim}\\
\Hom_{\uccC}([\rB,\rB],\rF^*) \ar[r]^{-\circ[\iota',\rB]} & \Hom_{\uccC}([\rC,\rB],\rF^*).
}
\]
Hence, the bottom row is an isomorphism which holds for any $1$-morphism $\rF$.
Thus, $[\iota',\rB]\colon[\rC,\rB]\to [\rB,\rB]$ is an isomorphism whose inverse we denote by $\alpha$.
Using that $[\rC,-]\cong \Id_{\comod_{\uccC}(\rC)}$ by Corollary \ref{idcohom} yields commutative diagram
$$\xymatrix{
[\rB,\rB]\ar^{\alpha }_{\sim}[r]& [\rC,\rB]\ar^{[\rC,\iota']}[r] \ar_{\sim}[d]&[\rC,\rC]\ar^{\sim}[d]\\
& \rB\ar@{^(->}^{\iota'}[r] & \rC,
}$$  with vertical isomorphisms. In particular, $[\rC,\iota']$ is mono.
So setting $\rD:=[\rB,\rB]$, we obtain a monomorphism $\iota:\rD \to \rC$ in $\underline{\cC}$.

Showing that $\rD\overset{\iota}{\hookrightarrow}\rC$ is a subcoalgebra is equivalent to showing that $[\rB,\rB] \overset{\theta}{\hookrightarrow}[\rC,\rC]$ is a subcoalgebra, where $\theta:=[\rC,\iota']\circ\alpha$.
For simplicity, let us denote by $\mu_\rC,\epsilon_\rC$ the comultiplication and counit of $[\rC,\rC]$ throughout the rest of the proof.

We first verify the compatibility of the counit maps of $\rD$ and $\rC$, i.e. $\epsilon_\rD=\epsilon_\rC\circ\theta$.
Using the definition $\theta=[\rC,\iota']\circ \alpha$ and that $\alpha$ is the inverse of $([\iota',\rB])^{-1}$, this is equivalent to showing that $\epsilon_\rC\circ_1[\rC,\iota'] = \epsilon_\rD \circ_1 [\iota',\rB]$.
Recall that, for any $\rX\in\comod_{\uccC}(\rC)$, the counit of $[\rX,\rX]$ is the map in $\ucC$ corresponding to $\id_\rX$ under the adjunction isomorphism $\Hom_{\uccC}([\rX,\rX],\one)\cong \Hom_{\comod_{\uccC}(\rC)}(\rX,\rX)$.
We consider the commutative diagrams
$$\xymatrix{
\Hom_{\uccC}([\rC,\rC],\one)\ar@{-}^{\sim}[r]\ar^{-\circ_1[\rC,\iota']}[d]& \Hom_{\comod_{\uccC}(\rC)} (\rC,\rC)\ar^{-\circ\iota'}[d]\\
\Hom_{\uccC}([\rC,\rB],\one) \ar@{-}^{\sim}[r]& \Hom_{\comod_{\uccC}(\rC)} (\rB,\rC)
}$$
and
$$\xymatrix{
\Hom_{\uccC}([\rB,\rB],\one)\ar@{-}^{\sim}[r]\ar^{-\circ_1[\iota', \rB]}[d]& \Hom_{\comod_{\uccC}(\rC)} (\rB,\rB)\ar^{\iota'\circ-}[d]\\
\Hom_{\uccC}([\rC,\rB],\one) \ar@{-}^{\sim}[r]& \Hom_{\comod_{\uccC}(\rC)} (\rB,\rC)
}$$

where the second is obtained from 
combining the natural transformation $[\iota',-]:[\rC,-]\to[\rB,-]$ with the adjoint pairs $([\rB,-],-\cdot\rB)$ and $([\rC,-],-\cdot\rC)$.
Since $\id_{\rC}\circ_1\iota' = \iota' \circ_1 \id_{\rB}$, and these two maps correspond to $\epsilon_\rC \circ_1[\rC,\iota']$ and $\epsilon_\rD \circ_1 [\iota',\rB]$ respectively on the left columns of the diagrams, the latter two maps are equal, as claimed.

To show compatibility of the comultiplications, let us start by recalling some essential facts.
For any $\rX,\rY\in\comod_{\uccC}(\rC)$, the coevaluation map $\mathrm{coev}_{\rX,\rY}:\rY\to[\rX,\rY]\rX$ is the map corresponding to $\id_{[\rX,\rY]}$ under the adjunction $\Hom_{\uccC}([\rX,\rY],[\rX,\rY])\cong \Hom_{\comod_{\uccC}(\rC)}(\rY,[\rX,\rY]\rX)$.
The comultiplication of the coalgebra $[\rX,\rX]$ is given by the map in $\ucC$ corresponding to $(\id_{[\rX,\rX]}\circ_0\mathrm{coev}_{\rX,\rX})\circ_1 \mathrm{coev}_{\rX,\rX}$.

Observe that the following diagram is commutative.
\[
\xymatrix@C=100pt@R=35pt{
\rB \ar[r]^{\mathrm{coev}_{\rB,\rB}} \ar[d]_{\id}
& [\rB,\rB]\rB \ar[r]^{\id_{[\rB,\rB]}\circ_0\mathrm{coev}_{\rB,\rB}} \ar[d]_{\id} 
& [\rB,\rB][\rB,\rB]\rB  \ar[d]_{\id_\rD\circ_0\alpha\circ_0\iota'}\\
\rB \ar[r]^{\mathrm{coev}_{\rB,\rB}} \ar[d]_{\id} & [\rB,\rB]\rB \ar[r]^{\id_{[\rB,\rB]}\circ_0\mathrm{coev}_{\rB,\rB}} \ar[d]_{\alpha\circ_0\iota'}& [\rB,\rB][\rC,\rB]\rC \ar[d]_{\alpha\circ_0[\rC,\iota']\circ_0\id_\rC}\\
\rB \ar[r]^{\mathrm{coev}_{\rC,\rB}} \ar[d]_{\iota'}& [\rC,\rB]\rC \ar[r]^{\id_{[\rC,\rB]}\circ_0\mathrm{coev}_{\rC,\rC}} \ar[d]_{[\rC,\iota']\circ_0\id_\rC}& [\rC,\rB][\rC,\rC]\rC \ar[d]_{[\rC,\iota']\circ_0\id_{[\rC,\rC]\rC}}\\
\rC \ar[r]^{\mathrm{coev}_{\rC,\rC}} & [\rC,\rC]\rC \ar[r]^{\id_{[\rC,\rC]}\circ_0\mathrm{coev}_{\rC,\rC}} & [\rC,\rC][\rC,\rC]\rC 
}
\]
Indeed, commutativity of the top left square is trivial; that of the bottom right square is easy, since both maps are just $[\rC,\iota']\circ_0 \mathrm{coev}_{\rC,\rC}$.
It is also easy to see that commutativity of the top (resp. middle) right square follows immediately from that of the middle (resp. bottom) left square as the former are obtained from the latter by horizontally composing with identity maps.

To see that the middle left square commutes (i.e. $\mathrm{coev}_{\rC,\rB}=(\alpha\circ_0\iota')\circ_1\mathrm{coev}_{\rB,\rB}$), we use the commutative diagrams
$$\xymatrix{
\Hom_{\uccC}([\rC,\rB],[\rC,\rB])\ar@{-}^{\sim}[r]\ar_{[\iota',\rB]\circ_1-}[d]& \Hom_{\comod_{\uccC}(\rC)} (\rB,[\rC,\rB]\rC)\ar^{([\iota',\rB]\circ_0\id_\rC)\circ-}[d]\\
\Hom_{\uccC}([\rC,\rB],[\rB,\rB]) \ar@{-}^{\sim}[r]& \Hom_{\comod_{\uccC}(\rC)} (\rB,[\rB,\rB]\rC)
}$$
and
$$\xymatrix{
\Hom_{\uccC}([\rB,\rB],[\rB,\rB])\ar@{-}^{\sim}[r]\ar_{-\circ_1[\iota',\rB]}[d]& \Hom_{\comod_{\uccC}(\rC)} (\rB,[\rB,\rB]\rB)\ar^{(\id_{[\rB,\rB]}\circ_0\iota')\circ-}[d]\\
\Hom_{\uccC}([\rC,\rB],[\rB,\rB]) \ar@{-}^{\sim}[r]& \Hom_{\comod_{\uccC}(\rC)} (\rB,[\rB,\rB]\rC),
}
$$
as well as $[\iota',\rB]\circ_1\id_{[\rC,\rB]} = \id_{[\rB,\rB]}\circ_1 [\iota',\rB]$. Together, these yield
\[
([\iota',\rB]\circ_0\id_\rC)\circ_1\mathrm{coev}_{\rC,\rB} = (\id_{[\rB,\rB]}\circ_0\iota')\circ_1\mathrm{coev}_{\rB,\rB},
\]
hence
$\mathrm{coev}_{\rC,\rB} = (\alpha\circ_0\id_\rC)\circ_1 (\id_{[\rB,\rB]}\circ_0\iota')\circ_1\mathrm{coev}_{\rB,\rB} = (\alpha\circ_0 \iota')\circ_1 \mathrm{coev}_{\rB,\rB}$.

Commutativity of the bottom left square (i.e. $\mathrm{coev}_{\rC,\rC}\circ_1\iota' = ([\rC,\iota']\circ_0\id_\rC)\circ_1 \mathrm{coev}_{\rC,\rB}$) follows similarly from the commutative diagrams
$$\xymatrix{
\Hom_{\uccC}([\rC,\rC],[\rC,\rC])\ar@{-}^{\sim}[r]\ar^{-\circ_1[\rC,\iota']}[d]& \Hom_{\comod_{\uccC}(\rC)} (\rC,[\rC,\rC]\rC)\ar^{-\circ\iota'}[d]\\
\Hom_{\uccC}([\rC,\rB],[\rC,\rC]) \ar@{-}^{\sim}[r]& \Hom_{\comod_{\uccC}(\rC)} (\rB,[\rC,\rC]\rC)
}$$
and 
$$\xymatrix{
\Hom_{\uccC}([\rC,\rB],[\rC,\rB])\ar@{-}^{\sim}[r]\ar^{[\rC,\iota']\circ_1-}[d]& \Hom_{\comod_{\uccC}(\rC)} (\rB,[\rC,\rB]\rC)\ar^{([\rC,\iota']\circ_0\id_\rC)\circ-}[d]\\
\Hom_{\uccC}([\rC,\rB],[\rC,\rC]) \ar@{-}^{\sim}[r]& \Hom_{\comod_{\uccC}(\rC)} (\rB,[\rC,\rC]\rC),
}$$
together with $[\rC,\iota']\circ_1\id_{[\rC,\rB]} = \id_{[\rC,\rC]}\circ_1[\rC,\iota']$.

Now that we know all six squares commute, composing the maps on the outer boundary of the big square yields
\begin{equation}\label{eq-bigsq}
\mu_\rC^\vee \circ_1\iota' = (\theta\circ_0\theta\circ_0\iota') \circ_1\mu_\rD^\vee,
\end{equation}
where $\mu_\rC^\vee := (\id_{[\rC,\rC]}\circ_0\mathrm{coev}_{\rC,\rC}) \circ_1\mathrm{coev}_{\rC,\rC}$ and $\mu_\rD^\vee := (\id_{[\rB,\rB]}\circ_0\mathrm{coev}_{\rB,\rB}) \circ_1\mathrm{coev}_{\rB,\rB}$ are the maps that correspond to $\mu_\rC$ and $\mu_\rD$ respectively under adjunction.

Using the commutative diagram
$$\xymatrix{
\Hom_{\uccC}([\rC,\rC],[\rC,\rC][\rC,\rC])\ar@{-}^{\sim}[r]\ar_{-\circ_1[\rC,\iota']}[d]& \Hom_{\comod_{\uccC}(\rC)} (\rC,[\rC,\rC][\rC,\rC]\rC)\ar^{-\circ\iota'}[d]\\
\Hom_{\uccC}([\rC,\rB],[\rC,\rC][\rC,\rC]) \ar@{-}^{\sim}[r]& \Hom_{\comod_{\uccC}(\rC)} (\rB,[\rC,\rC][\rC,\rC]\rC),
}$$
we can see that the left-hand map $\mu_\rC^\vee \circ_1\iota'$ of \eqref{eq-bigsq} corresponds to $\mu_\rC\circ_1[\rC,\iota']$ under the adjunction isomorphism of the bottom row.

We claim that the right-hand map $(\theta\circ_0\theta\circ_0\iota')\mu_\rD^\vee$ of \eqref{eq-bigsq} corresponds to $(\theta\circ_0\theta)\circ_1\mu_\rD\circ_1[\iota',\rB]$ under the same adjunction isomorphism.
Indeed, using the commutative diagram
$$
\xymatrix{
\Hom_{\uccC}([\rB,\rB],[\rB,\rB][\rB,\rB])\ar@{-}^{\sim}[r]\ar_{-\circ_1[\iota',\rB]}[d]& \Hom_{\comod_{\uccC}(\rC)} (\rB,[\rB,\rB][\rB,\rB]\rB)\ar^{(\id_{[\rB,\rB][\rB,\rB]}\circ_0\iota')\circ_1-}[d]\\
\Hom_{\uccC}([\rC,\rB],[\rB,\rB][\rB,\rB]) \ar@{-}^{\sim}[r]\ar_{(\theta\circ_0\theta)\circ_1-}[d]& \Hom_{\comod_{\uccC}(\rC)} (\rB,[\rB,\rB][\rB,\rB]\rC) \ar^{(\theta\circ_0\theta\circ_0\id_\rC)\circ_1-}[d] \\
\Hom_{\uccC}([\rC,\rB],[\rC,\rC][\rC,\rC]) \ar@{-}^{\sim}[r]& \Hom_{\comod_{\uccC}(\rC)} (\rB,[\rC,\rC][\rC,\rC]\rC),
}
$$
the correspondence between $\mu_\rD$ and $\mu_\rD^\vee$ on the top row induces the correspondence between
$\mu_D\circ_1[\iota',\rB]$ and $(\id_{[\rB,\rB][\rB,\rB]}\circ_0\iota')\circ_1\mu_\rD^\vee$ on the second row, which in turn induces a correspondence between $(\theta\circ_0\theta)\circ_1\mu_D\circ_1[\iota',\rB]$ and $(\theta\circ_0\theta\circ_0\id_\rC)\circ_1 (\id_{[\rB,\rB][\rB,\rB]}\circ_0\iota')\circ_1\mu_\rD^\vee = (\theta \circ_0\theta\circ_0\iota')\circ_1\mu_\rD^\vee$ on the bottom row.

Thus, \eqref{eq-bigsq} is equivalent to saying that $\mu_\rC\circ_1[\rC,\iota'] =(\theta\circ_0\theta)\circ_1\mu_D\circ_1[\iota',\rB]$.
Since $\theta = [\rC,\iota']\circ_1([\iota',\rB])^{-1}$, we obtain that $\mu_\rC\circ_1\theta =(\theta\circ_0\theta)\circ_1\mu_D$.
This completes the proof of the compatibility between comultiplications of $\rD$ and $[\rC,\rC]\cong \rC$ under $\theta$.

It remains to show the equivalence $-\square_{\rD}\rD_\rC:\comod_{\uccC}(\rD)\to\S$.
For a $\rD$-comodule $\rM$, we have an exact sequence $0\to \rM\to \rF\rD$ in $\comod_{\uccC}(\rD)$ for some $\rF\in \cC$.
Recall that $ip(\rC)=\rB\cong [\rC,\rB] \cong [\rB,\rB]$, so we have isomorphisms of right $\rC$-comodules $ip(\rC)\cong \rD_\rC \cong \rC\square_\rC\rD\square_\rD\rD_\rC$.  In particular, we have $\rD_\rC\in \S$.
Since $\S$ is $\cC$-stable, 
we have $(\rF\rD)\square_\rD\rD_\rC\cong \rF\rD_\rC\in \S$, so it follows from the assumption of $\S$ being closed under subobjects that $\rM\square_\rD\rD_\rC\in \S$.
Hence, $-\square_\rD\rD_\rC$ induces a well-defined functor from $\comod_{\uccC}(\rD)$ to $\S$.

Recall from Lemma \ref{lemclosed} that $-\square_{\rD}\rD_\rC$ is fully faithful.  It remains to show that it is dense.  Indeed, if $\rM\in \S$, then we have an exact sequence $0\to \rM\to \rF\rC$ in $\comod_{\uccC}(\rC)$ for some $\rF\in\cC$, which induces an exact sequence $0\to ip(\rM)\to ip(\rF\rC)$.  By assumption, we have $ip(\rM)=\rM$.  
Since by assumption $i(\rF\rM)\cong \rF i(\rM)$ for all $\rM\in\S$, we also have $p(\rF\rN)\cong \rF p(\rN)$ for all $\rN\in\comod_{\uccC}(\rC)$, as can be seen from the chain of isomorphisms
\begin{equation*}
\begin{split}
\Hom_{\S}(\rM, p(\rF\rN))&\cong \Hom_{\comod_{\uccC}(\rC)}(i(\rM), \rF\rN)\\
&\cong \Hom_{\comod_{\uccC}(\rC)}({}^*\rF i(\rM), \rN)\\
&\cong \Hom_{\comod_{\uccC}(\rC)}( i({}^*\rF\rM), \rN)\\
& \cong \Hom_{\S}( {}^*\rF\rM, p(\rN))\\
& \cong \Hom_{\S}( \rM, \rF p(\rN)),\\
\end{split}
\end{equation*}
which holds for any $\rM\in\S$.
We thus have $ip(\rF\rC)\cong \rF (ip(\rC))\cong \rF\rD$, which is in the essential image of $-\square_\rD\rD_\rC$.  Thus, as $\comod_{\uccC}(\rD)$ is closed under subobjects by Lemma \ref{lemclosed} and $-\square_\rD\rD_\rC$ is exact, $\rM$ is also in the essential image of $-\square_{\rD}\rD_\rC$.
\end{proof}

This leads us to the following proposition, which generalises \cite[Theorem 4.2(iii)]{NT}.

\begin{proposition}\label{cat-coalg-bij}
The construction in Lemma \ref{realise-subcat} induces a bijection between the set of $\cC$-stable subobject-closed quotient-closed full subcategories of $\comod_{\uccC}(\rC)$ up to equivalence and the set of subcoalgebras of $\rC$ up to isomorphism.
\end{proposition}
\begin{proof}
Let $\Omega$ be the set of $\cC$-stable subobject-closed quotient-closed full subcategories of $\comod_{\uccC}(\rC)$ up to equivalence, and $\Phi$ be the set of subcoalgebras of $\rC$ up isomorphism.
By Lemma \ref{realise-subcat}, assigning $\S \mapsto [ip(\rC),ip(\rC)]$, where $i$ is the inclusion of $\S$ into $\comod_{\uccC}(\rC)$ and $p$ is the right adjoint of $i$, defines a map $f:\Omega\to\Phi$.

On the other hand, for a subcoalgebra $\rD$, it follows from Lemma \ref{lemclosed} that $\comod_{\uccC}(\rD)$ is equivalent to a subobject-closed quotient-closed full subcategory of $\comod_{\uccC}(\rC)$.  Note that this subcategory is also $\cC$-stable as $\rD$ is a coalgebra $1$-morphism in $\ucC$.
Clearly, isomorphic subcoalgebras define the same full subcategory up to equivalence.
Hence, we have a map $g:\Phi\to \Omega$.

Starting with $\S\in \Omega$, we have $gf(\S)=\comod_{\uccC}(f(\S))$, which is equivalent to $\S$ by Lemma \ref{realise-subcat}; this means that $gf=\id_\Omega$.
For $\rD\in \Phi$, Lemma \ref{MMMZCor7Lem8} says that the inclusion of $\comod_{\uccC}(\rD)$ into $\comod_{\uccC}(\rC)$ and its right adjoint are given by $-\square_\rD\rD_\rC$ and $-\square_\rC\rD$ respectively.
Since $\rC\square_\rC\rD\square_\rD\rD_\rC\cong \rD_\rC$, the subcoalgebra $fg(\rD)$ is given $[\rD_\rC,\rD_\rC]$.
By the same argument as in the first two paragraphs in the proof of Lemma \ref{realise-subcat}, we have $[\rD_\rC,\rD_\rC]\cong [\rC,\rD] \cong \rD$.
Therefore, we have $fg(\rD)\cong \rD$, i.e. $fg=\id_\Phi$ as required.
\end{proof}

\section{Coidempotent subalgebras and extensions}\label{coidext}

\subsection{Coidempotent subcoalgebras.}
Following \cite{NT}, we define the following notion, which, in the classical setting, is dual to idempotent quotient algebras $A/AeA$.

\begin{definition}\label{def:coidem}
Let $\rC$ be a coalgebra $1$-morphism in $\underline{\cC}$ and $\rD$ a subcoalgebra of $\rC$. We say that  $\rD$ an {\bf coidempotent subcoalgebra} of $\rC$ if $\mu_{\rC}^{-1}(\rC\rD+\rD\rC)=\rD$ or, equivalently, for $\rJ=\rC/\rD$, the map $\mu_\rJ:=(\id_\rJ\circ_0\pi_\rJ)\circ_1\rho_\rJ:\rJ\to \rJ\rJ$ is a monomorphism in $\ucC$, where $\pi_\rJ:\rC\to\rJ$ is the natural projection and $\rho_\rJ$ is the right $\rC$-coaction map of $\rJ$.
\end{definition}

\begin{lemma}\label{seqzero} 
Let $\rC$ be a coalgebra $1$-morphism in $\underline{\cC}$ and $\rD$ a subcoalgebra. Set $\rJ=\rC/\rD$ and let $\rI$ be the injective hull of $\rJ$ in $\comod_{\underline{\ccC}}(\rC)$. 
\begin{enumerate}[$($i$)$]
\item\label{seqzero1} $\rD$ is coidempotent if and only if $\rJ\square_{\rC}\rD=0$.
\item\label{seqzero2} For $\rQ\in \inj_{\underline{\ccC}}(\rC)$, if $\rQ\square_{\rC}\rD=0$, then $\rQ\in \add_{\inj_{\underline{\ccC}}(\rC)}\{\rF\rI\mid \rF\in\cC\}$.  Moreover, the converse holds when $\rD$ is coidempotent.
\end{enumerate}
\end{lemma}

\proof
\eqref{seqzero1} 
Applying $\rJ\square_\rC -$ to the exact sequence $0\to\rD\to\rC\overset{\pi}{\to}\rJ\to 0$ of $\rC$-$\rC$-bicomodules yields and exact sequence
$$0 \longrightarrow \rJ\square_\rC\rD\longrightarrow\rJ\square_\rC\rC\overset{\id_\rJ\square\pi}{\longrightarrow}\rJ\square_\rC\rJ.$$
Now consider the diagram
$$\xymatrix{
\rJ\square_\rC\rC \ar@{^{(}->}_{\alpha}[d] \ar^{\id_\rJ\square\pi}[r]& \rJ\square_\rC\rJ\ar@{^{(}->}^{\alpha'}[d]\\
\rJ\rC \ar^{\id_\rJ\circ_0\pi}[r] \ar_{\id_\rJ\circ_0\mu_\rC-\rho_\rJ\circ_0\id_\rC}[d] & \rJ\rJ \ar^{\id_\rJ\lambda_\rJ-\rho_\rJ\circ_0\id_\rJ}[d] \\
\rJ\rC\rC \ar^{\id_{\rJ\rC}\circ_0\pi}[r]&\rJ\rC\rJ,
}$$
where $\lambda_\rJ$ is the left $\rC$-coaction map of $\rJ$.
Using the interchange law and the induced (left) $\rC$-comodule structure of $\rJ$, the lower square commutes, which yields the commutativity of the upper square.
Since there is an isomorphism $\beta:\rJ \xrightarrow{\sim}\rJ\square_\rC\rC$, we have $\alpha\circ_1\beta=\rho_\rJ$.
The induced map $\mu_\rJ\colon\rJ\to\rJ\rJ$ is precisely $(\id_\rJ\circ_0\pi)\circ_1\rho_\rJ$. 
Hence, we have two exact sequences
$$
\xymatrix{
0\ar[r]& \rJ\square_\rC\rD \ar[r]\ar@{.>}[d]  & \rJ\square_\rC\rC \ar_{\beta^{-1}}^{\sim}[d] \ar^{\id_\rJ\square\pi}[r]& \rJ\square_\rC\rJ\ar@{^{(}->}^{\alpha'}[d] \\
0\ar[r]& \ker\mu_\rJ \ar[r] & \rJ \ar^{\mu_\rJ}[r] & \rJ\rJ
}
$$
so that the right-hand square commutes.
This implies that $\rJ\square_\rC\rD\cong \ker\mu_\rJ$. The claim follows.


\eqref{seqzero2}
Realise $\rQ\in \inj_{\underline{\ccC}}(\rC)$ as a direct summand (inside $\inj_{\underline{\ccC}}(\rC)$) of $\rG \rC$, for some $1$-morphism $\rG\in \cC$, with complement $\rQ'$. Let $-\square_\rD \rD_\rC\colon \comod_{\uccC}(\rD)\hookrightarrow  \comod_{\uccC}(\rC)$ be the morphism from Lemma \ref{lemclosed} given by extending the coaction from $\rD$ to $\rC$.

Consider the exact sequence
$$0\to \rG\rD\to \rQ\oplus \rQ'\to\rG \rJ$$
in $\comod_{\uccC}(\rC)$.
We claim that if the induced morphism $\alpha\colon\rG \rD\to \rQ$ is nonzero, then $\rQ\square_\rC\rD\neq 0$. 
Indeed, as $\rG\rD$ is in the essential image of $-\square_\rD \rD_\rC$, the nonzero image $\rZ$ of $\alpha$, as a quotient of $\rG\rD$, is also in the essential image of $-\square_\rD \rD_\rC$ by Lemma \ref{lemclosed}, and isomorphic to $\rZ'\square_\rD \rD_\rC$ for some $\rZ'\in \comod_{\uccC}(\rD)$.
On the one hand, applying $-\square_{\rC}\rD$ to the monomorphism $\rZ\hookrightarrow  \rQ$ yields a monomorphism $\rZ\square_\rC\rD\hookrightarrow  \rQ\square_\rC\rD$.
On the other hand, it follows from Lemma \ref{MMMZCor7Lem8}\eqref{MMMZ7} that $Z\square_{\rC}\rD \cong Z'\square_{\rD}\rD\square_{\rC}\rD \cong Z'$ is nonzero.
Thus we obtain that $\rQ\square_{\rC}\rD$ is also nonzero, as claimed.

Therefore, if $\rQ\square_\rC\rD=0$, then $\rQ$ is not in the coimage of the first map of the exact sequence above.  This implies that $\rQ$ is isomorphic to a subobject of $\rG\rJ$, which in turn is a subobject of $\rG \rI$. Injectivity of $\rQ$ implies that it is in fact isomorphic to a direct summand of $\rG \rI$.

Let us now assume $\rD$ is coidempotent and show the converse.
Let $\rF_0$ be the injective hull of $\rJ$ in $\ucC$ and $\vartheta\colon \rJ\hookrightarrow \rF_0$ the canonical embedding.
Since the induced comultiplication on $\rJ$ is, by assumption, a monomorphism in $\ucC$ and composition in $\ucC$ is left exact, we have monomorphisms $\rJ\hookrightarrow \rJ\rJ \hookrightarrow \rF_0\rJ$ in $\comod_{\uccC}(\rC)$. 
We obtain a commutative diagram
$$\xymatrix{
\rJ \ar@{^{(}->}_{\mu_\rJ}[rr] \ar@{^{(}->}[d]^{\rho_\rJ}&&\rJ \rJ \ar@{^{(}->}^{\vartheta\circ_0 \id_\rJ}[d]\\
\rJ \rC \ar^{\id_\rJ\circ_0\pi}[urr]\ar@{^{(}->}_{\vartheta\circ_0 \id_\rC}[d]&&\rF_0\rJ \\
\rF_0\rC\ar@{->>}^{\id_{\rF_0}\circ_0 \pi}[urr]&&
}$$
in $\comod_{\uccC}(\rC)$. 
By injectivity of $\rF_0\rC$, the resulting maps $\tau=(\vartheta\circ_0 \id_\rJ)\circ_1\mu_\rJ$ and $\sigma=(\vartheta\circ_0 \id_\rC)\circ_1\rho_\rJ$ in the diagram 
$$\xymatrix{
\rJ\ar@{^{(}->}^{\tau}[rr]\ar@{^{(}->}_{\sigma}[d]&&\rF_0\rJ \ar@{-->}@/_/_{\kappa}[dll]\\
\rF_0\rC\ar@{->>}@/_/_{\id_{\rF_0}\circ_0\pi}[urr]&&
}$$
give rise to the dotted map $\kappa\colon\rF_0\rJ \to \rF_0\rC$, such that the diagram commutes both ways around. The equality $\kappa\circ_1(\id\circ_0\pi)\circ_1\sigma=\kappa\tau=\sigma$ implies that $\kappa\circ_1(\id\circ_0\pi)$ is the identity on $\rI$ as a direct summand of $\rF_0\rC$ and hence $\rI$ is a direct summand of $\rF_0\rJ$.
By part \eqref{seqzero1}, we have $\rF_0\rJ\square_{\rC}\rD=0$.
In particular, its direct summand $\rI\square_{\rC}\rD$ is also zero, and hence any $Q\in\add_{\inj_{\underline{\ccC}}(\rC)}\{\rF\rI\mid \rF\in\cC\}$ satisfies $\rQ\square_{\rC}\rD=0$.
\endproof

\begin{lemma}\label{kill-simple}
Suppose $\rD\overset{\iota}{\hookrightarrow}\rC$ is a coidempotent subcoalgebra.
Let $\rI$ be the injective hull of the cokernel of $\iota$ and $\rM$ be a simple $\rC$-comodule with injective hull $\rQ$.
Then the following are equivalent.
\begin{enumerate}[(i)]
\item\label{kill-simple1} $\rM\square_\rC\rD=0$.
\item\label{kill-simple2} $\rQ\square_\rC\rD=0$.
\item\label{kill-simple3} $\rQ\in \add\{\rF\rI\mid \rF\in\cC\}$.
\end{enumerate}
\end{lemma}
\begin{proof}
\eqref{kill-simple2}$\Leftrightarrow$\eqref{kill-simple3}: This is Lemma \ref{seqzero} \eqref{seqzero2}.

\eqref{kill-simple2}$\Rightarrow$\eqref{kill-simple1}: Clear by left exactness of $-\square_\rC\rD$.

\eqref{kill-simple1}$\Rightarrow$\eqref{kill-simple2}:
By Lemma \ref{MMMZCor7Lem8} \eqref{MMMZ8}, $\rQ\square_\rC\rD\square_\rD\rD_\rC$ is a subcomodule of $\rQ$, which has simple socle $\rM$ in the case when it is non-zero.
Since the smallest non-trivial subcomodule $\rM$ of $\rQ$ is annihilated by $-\square_\rC\rD$, it follows that $\rQ\square_\rC\rD\square_\rD\rD_\rC=0$.
But $-\square_\rD\rD_\rC$ is fully faithful, so $\rQ\square_\rC\rD=0$.
\end{proof}

\subsection{Coidempotent subcoalgebras and Serre subcategories}

In this subsection, we provide a correspondence between coidempotent subcoalgebras of a coalgebra $1$-morphism $\rC$ and Serre subcategories of $\comod_{\uccC}(\rC)$. Throughout this subsection, we let $\rD\overset{\iota}{\hookrightarrow}\rC$ be a subcoalgebra, let $\rJ,\pi_\rJ$ be defined by the short exact sequence
$$0\to\rD\overset{\iota}{\hookrightarrow}\rC\overset{\pi_\rJ}{\twoheadrightarrow}\rJ\to 0$$
and $\mu_\rJ=(\id_\rJ\circ_0\pi_\rJ)\circ_1\rho_\rJ$ the induced multiplication on $\rJ$.

\begin{lemma}\label{coidem-then-Serre}
If $\rD\overset{\iota}{\hookrightarrow}\rC$ is a coidempotent subcolagebra, then the fully faithful exact embedding $-\square_\rD\rD_\rC$ sends $\comod_{\uccC}(\rD)$ to a Serre subcategory of $\comod_{\uccC}(\rC)$.
\end{lemma}
\proof
By Lemma \ref{lemclosed}, it remains to show closure under extensions.

For any $M\in \comod_{\uccC}(\rC)$, we denote by $\sigma_M$ the composition $(\id_M\circ_0\pi_\rJ)\circ_1\rho_M$, where $\rho_M$ is the coaction map.
Then $M$ being in the essential image of $\comod_{\uccC}(\rC)$ is equivalent to $\sigma_M=0$.
Let $0\to X\xrightarrow{f} Y\xrightarrow{g}Z \to 0$ be a short exact sequence of right $\rC$-comodule such that $X,Z$ is in the essential image of $-\square_\rD\rD_\rC$.
Our aim is to show that $\sigma_Y$ is zero.

Since horizontal composition is left exact, we have commutative diagram 
$$
\xymatrix@C=45pt{
0\ar[r] & X\ar[r]^{f}\ar[d]_{\rho_X} & Y\ar[d]_{\rho_Y} \ar[r]^{g} & Z \ar[d]^{\rho_Z} \ar[r] & 0\\
0\ar[r] & X\rC \ar[r]_{f\circ_0\id_\rC} & Y\rC \ar[r]_{g\circ_0\id_\rC} & Z\rC 
}
$$
in $\ucC$ with exact rows.

This induces a commutative diagram where all $\rC$'s and $\rho$'s above are replaced by $\rJ$ and $\sigma$ respectively.
Hence we have $(g\circ_0\id_\rJ)\circ_1\sigma_Y = \sigma_Z\circ_1g = 0$, which means that the image of $\sigma_Y$ is in the kernel of $g\circ_0\id_\rJ$.
Exactness of the top row of the diagram implies that there is $\phi:Y\to X\rJ$ so that $(f\circ_0\id_\rJ)\circ_1\phi = \sigma_Y$.
Thus, we have 
\begin{equation*}\begin{split}
(\sigma_Y\circ_0\id_\rJ)\circ_1 \sigma_Y& = (\sigma_Y\circ_0\id_\rJ)\circ_1(f\circ_0\id_\rJ)\circ_1\phi \\
&= ((\sigma_Y\circ_1 f)\circ_0\id_\rJ)\circ_1\phi\\
&= (((f\circ_0\id_\rJ)\circ_1\sigma_X)\circ_0 \id_\rJ) \circ_1\phi \\
&=0.
 \end{split}\end{equation*}

On the other hand, $(Y\xrightarrow{\rho_Y}Y\rC\xrightarrow{\rho_Y\circ_0\id_\rC}Y\rC\rC )=(Y\xrightarrow{\rho_Y}Y\rC\xrightarrow{\id_Y\circ_0\mu_\rC}Y\rC\rC )$ and this induces $(\sigma_Y\circ_0\id_\rJ)\circ_1 \sigma_Y = (\id_Y\circ_0\mu_\rJ)\circ_1\sigma_Y$. Combining this with the argument in the previous paragraph, we see that $(\id_Y\circ_0\mu_\rJ)\circ_1\sigma_Y=0$.  Since $\rD$ is coidempotent (i.e. $\mu_\rJ$ is mono) and horizontal left composition with $\id_Y$ preserves monicity, we obtain that $\id_Y\circ_0\mu_\rJ$ is mono, which implies $\sigma_Y=0$ as required.
\endproof

\begin{lemma}\label{I/I^2 is A/I-module}
Let $\rK$ be the kernel of $\mu_\rJ$. 
Then there is a short exact sequence 
$$
0 \to \rD \to \mu_{\rC}^{-1}(\rC\rD+\rD\rC) \to \rK \to 0
$$
of right $\rC$-comodules, and $\rK$ is also a $\rD$-comodule.
\end{lemma}
\proof
The right $\rC$-comodule map $\pi_\rJ:\rC\twoheadrightarrow\rJ$ induces a commutative diagram
$$
\xymatrix@C=50pt{
\rC \ar[r]^{\mu_\rC}\ar[d]_{\pi_\rJ} & \rC\rC \ar[d]^{\pi_\rJ\circ_0\id_\rC} \ar[r]^{\id_\rC\circ_0\pi_\rJ} & \rC\rJ \ar[d]^{\pi_\rJ\circ_0\id_\rJ}\\
\rJ\ar[r]_{\rho_\rJ} &\rJ\rC \ar[r]_{\id_\rJ\circ_0\pi_\rJ}& \rJ\rJ ,
}
$$
in $\comod_{\uccC}(\rC)$.  Since $\mu_\rJ:=(\id_\rJ\circ_0\pi_\rJ)\circ_1\rho_\rJ$, $\mu_{\rC}^{-1}(\rC\rD+\rD\rC)=\ker((\pi_\rJ\circ_0\pi_\rJ)\circ_1\mu_\rC)$ coincides with $\ker(\mu_\rJ\circ_1\pi_\rJ)$.

We have a commutative diagram
$$
\xymatrix{
 & 0\ar[r]\ar[d] & \rC\ar[r]^{\id_\rC}\ar[d]^{\pi_\rJ} & \rC\ar[r]\ar[d]^{(\pi_\rJ\circ_0\pi_\rJ)\circ_1\mu_\rC} & 0\\
0 \ar[r] & \rK \ar[r]_{\iota_\rK} & \rJ \ar[r]_{\mu_\rJ} & \rJ\rJ & 
}
$$
where both rows are exact.
Now the snake lemma provides the required short exact sequence $0\to \rD\to  \mu_{\rC}^{-1}(\rC\rD+\rD\rC)\to \rK\to 0$ of right $\rC$-comodules.

Consider the following commutative diagram 
$$
\xymatrix@C=50pt{
 \rK \ar[r]^{\iota_\rK}\ar[d]_{\rho_\rK} & \rJ \ar[d]^{\rho_\rJ} \\
 \rK\rC \ar[r]^{\iota_\rK\circ_0\id_\rC}\ar[d]_{\id_\rK\circ_0\pi_\rJ}& \rJ\rC \ar[d]^{\id_\rJ\circ_0\pi_\rJ} \\
 \rK\rJ \ar[r]_{\iota_\rK\circ_0\id_\rJ} & \rJ\rJ
}
$$
in $\comod_{\uccC}(\rC)$.
This yields $$(\iota_\rK\circ_0\id_\rJ)\circ_1(\id_\rK\circ_0\pi_\rJ)\circ_1\rho_\rK = (\id_\rJ\circ_0\pi_\rJ)\circ_1\rho_\rJ\circ_1\iota_\rK = \mu_\rJ\circ_1\iota_\rK = 0.$$
In particular, since $\iota_\rK\circ_0\id_\rJ$ is mono (as, again, horizontal composition inherits monicity of $\iota_\rK$), we deduce that $(\id_\rK\circ_0\pi_\rJ)\circ_1\rho_\rK=0$, as required to show that $\rK$ is indeed a right $\rD$-comodule.
\endproof

\begin{remark}
All maps in the above proof are in fact morphisms of $\rC$-$\rC$-bicomodules, so the exact sequence in the statement of the lemma can be interpreted as an exact sequence of $\rC$-$\rC$-bicomodules. A similar proof shows that $\rK$ is also a $\rD$-$\rD$ bicomodule.
\end{remark}

\begin{proposition}\label{NT-corresp}
Suppose $\rD\overset{\iota}{\hookrightarrow}\rC$ is a subcoalgebra.
Then the fully faithful embedding $-\square_\rD\rD_\rC$ sends $\comod_{\uccC}(\rD)$ to a Serre subcategory of $\comod_{\uccC}(\rC)$ if, and only if, $\rD$ is coidempotent. 
\end{proposition}
\begin{proof}
If $\rD$ is coidempotent, we have already shown in Lemma \ref{coidem-then-Serre} that $\comod_{\uccC}(\rD)$ embeds as a Serre subcategory.  It remains to show the converse.

Recall from Lemma \ref{I/I^2 is A/I-module} that we have a short exact sequence of right $\rC$-comodules
$$
0\to \rD \to \rD_2 \to \rK \to 0,
$$
with $\rD_2=\mu_{\rC}^{-1}(\rC\rD+\rD\rC)$. Furthermore, $\rD$ and $\rK$ are both right $\rD$-comodules, meaning their $\rC$-coaction map factors through $\id\circ_0\iota$, that is, $\rD, \rK$ are in the essential image of $-\square_\rD\rD_\rC$.

Since a Serre subcategory is extension-closed, we obtain that $\rD_2$ is in the essential image of $-\square_\rD\rD_\rC$.
Note that $\rC\square_\rC\rD\square_\rD\rD_\rC=\rD_\rC $ is the maximal subobject of $\rC$ that belongs to the essential image of $-\square_\rD\rD_\rC$.
However, $\rD_2$ is a subobject of $\rC$ (the cokernel being the image of $\mu_\rJ$), so we deduce that $\rD_2\cong \rD$, i.e. $\rD$ is coidempotent.
\end{proof}

\subsection{Coidempotent subcoalgebras and recollements}

We have now shown $\cC$-stable Serre subcategories of $\comod_{\uccC}(\rC)$ can be associated to a coidempotent subcoalgebra of $\rD$.
It is natural to ask what the quotient $\comod_{\uccC}(\rC)/\comod_{\uccC}(\rD)$ is, or how the results in the previous subsection fit into the framework of recollements.

\begin{lemma}\label{ker-comod}
Let $\rI$ be an injective $\rC$-comodule.
The following hold.
\begin{enumerate}[$($i$)$]
\item\label{ker-comod0} Let $\rM$ be a simple $\rC$-comodule with injective hull $\rQ$, then $[\rI,\rM]=0$ if and only if $\rQ$ is not in $\add\{\rF\rI\mid\rF\in\cC\}$.

\item\label{ker-comod1} The full subcategory of $\comod_{\uccC}(\rC)$ given by the $\rC$-comodules $\rM$ with $[\rI,\rM]=0$ is equivalent to $\comod_{\uccC}(\rD)$ for some coidempotent subcoalgebra $\rD$ of $\rC$.
\item\label{ker-comod2} Let $\rD$ be the subcoalgebra of $\rC$ given in \eqref{ker-comod1}, and $\rM$ a simple $\rD$-comodule.
Then $\rM\square_{\rD}\rD_\rC$ is a simple $\rC$-comodule whose injective hull is not in $\add\{\rF\rI\mid\rF\in\cC\}\subset \inj_{\uccC}(\rC)$.\end{enumerate}
\end{lemma}
\begin{proof}
\eqref{ker-comod0}: By the defining property of internal homs, $[\rI,\rM]=0$ is equivalent to $\Hom_{\comod_{\uccC}(\rC)}(\rM,\rF\rI)\cong\Hom_{\uccC}([\rI,\rM],\rF)=0$ for all $\rF\in \ucC$.
This is the same as saying that $\rM$ is not in the socle of any of the object in $\add\{\rF\rI\mid \rF\in\cC\}$.

\eqref{ker-comod1}: Let $\rA$ be the coalgebra $1$-morphism given by $[\rI,\rI]$.  Then  $[\rI,-]:\comod_{\uccC}(\rC)\to\comod_{\uccC}(\rA)$ is exact by Lemma \ref{12.8}.  The full subcategory in the claim is then the kernel of an exact functor, hence a Serre subcategory.    This subcategory is clearly $\cC$-stable as $[\rI,-]$ is a morphism of 2-representations.
Now it follows from Lemma \ref{realise-subcat} that this category is equivalent to $\comod_{\uccC}(\rD)$ for some subcoalgebra $\rD$ of $\rC$, and $\rD$ being coidempotent follows from Lemma \ref{NT-corresp}.

\eqref{ker-comod2}: Since $\comod_{\uccC}(\rD)$ embeds (via $-\square_\rD\rD_\rC$) as a Serre subcategory of $\comod_{\uccC}(\rC)$, $\rM\square_{\rD}\rD_\rC$ is a simple $\rC$-comodule.  By the defining property of this Serre subcategory, $[\rI, \rM\square_\rD\rD_\rC]=0$, and the claim follows from \eqref{ker-comod0}.
\end{proof}

For a subcoalgebra $\rD$ of $\rC$, Lemma \ref{adjlemma} tells us that there is an adjoint triple
$$([{}_\rD\rD_\rC,-],-\square_\rD\rD_\rC \cong [{}_\rC\rD_\rD,-],-\square_\rC\rD)$$ between the comodule categories of these two coalgebras.
It follows from Lemma \ref{lemclosed} that $-\square_\rD\rD_\rC$ is fully faithful.

On the other hand, if we pick an injective $\rC$-comodule $\rI$ and let $\rA$ to be the coalgebra $1$-morphism given by $[\rI,\rI]$, then we obtain another adjoint triple
$$([[\rI,\rC],-],-\square_\rC [\rI,\rC]\cong [\rI,-],-\square_\rA\rI)$$
between the comodule categories of $\rC$ and $\rA$.
Note that the middle isomorphisms follow from Lemma \ref{12.8}. 
Moreover, $-\square_\rA\rI$ is fully faithful; one can see this by showing $[\rI,-]\circ (-\square_{\rA}\rI)$ is naturally isomorphic to the identity functor on $\comod_{\uccC}(\rA)$.
Indeed, as $[\rI,-]\cong -\square_{\rC}[\rI,\rC]$, the functor is naturally isomorphic to $-\square_{\rA}\rI\square_{\rC}[\rI,\rC] \cong -\square_{\rA}[\rI,\rI]=-\square_{\rA}\rA \cong \mathrm{Id}_{\comod_{\uccC}(\rA)}$, where the first isomorphism uses Lemma \ref{Iinfromleft}.

Similarly, we show that $[\rI,[[\rI,\rC],-]]\cong \mathrm{Id}_{\comod_{\uccC}(\rA)}$ to demonstrate that $[[\rI,\rC],-]$ is fully faithful.  To this end, we compute, for all $\rM,\rM'\in \comod_{\uccC}(\rA)$, that
 \begin{equation*}\begin{split}
 \Hom_{\comod_{\uccC}(\rA)}([\rI,[[\rI,\rC],\rM]], \rM')&\cong \Hom_{\comod_{\uccC}(\rC)}([[\rI,\rC],\rM]], \rM'\square_\rA \rI)\\
 &\cong \Hom_{\comod_{\uccC}(\rA)}(\rM, \rM'\square_\rA \rI\square_\rC[\rI,\rC])\\
 &\cong \Hom_{\comod_{\uccC}(\rA)}(\rM, \rM'),\\
 \end{split}\end{equation*}
where the last isomorphism uses the same argument as in the previous paragraph.

Suppose $\rC$ and either one of $\rD$ or $\rI$ is given. We would like to understand when the two adjoint triples above defines a recollement
\begin{equation}\label{recollement}
\xymatrix{  \comod_{\uccC}(\rD)\ar^{-\square_\rD\rD_\rC}[rr]&& \comod_{\uccC}(\rC)\ar^{[\rI,-]}[rr]\ar@<1ex>@/^/^{-\square_\rC\rD}[ll]\ar@<-2ex>@/_/_{[\rD,-]}[ll]&& \comod_{\uccC}(\rA).\ar@<1ex>@/^/^{-\square_\rA\rI}[ll]\ar@<-2ex>@/_/_{[[\rI,\rC],-]}[ll]
}\end{equation}
of comodule categories. In other words, we ask under what conditions on $\rD$ and $\rI$,  $\comod_{\uccC}(\rD)$ embeds via $-\square_\rD\rD_\rC$ as a Serre subcategory and coincides with the kernel (category) of the exact functor $[\rI,-]$.

\begin{proposition}\label{recollcondition}
Suppose $\rD\overset{\iota}{\hookrightarrow} \rC$ is a coidempotent subcoalgebra, and $\rI$ the injective hull in $\comod_{\uccC}(\rC)$ of the cokernel of $\iota$.
Then we have a recollement of the form \eqref{recollement}.
\end{proposition}

\proof
We already know from Proposition \ref{NT-corresp} that the essential image of $\comod_{\uccC}(\rD)$ under  $-\square_\rD\rD_\rC$ is a Serre subcategory of $\comod_{\uccC}(\rC)$.  It remains to show that this coincides with the full subcategory consisting of $\rM\in\comod_{\uccC}(\rC)$ such that $[\rI,\rM]=0$.
It suffices to check that $[\rI,\rM]=0 \in \comod_{\uccC}(\rA)$ if and only if $\rM\cong \rM\square_\rC\rD\square_\rD\rD_\rC$.  Furthermore, it suffices to check that these conditions are equivalent for simple objects $\rM$.
 
Let $\rQ$ be the injective hull of $\rM$.
It follows from Lemma \ref{ker-comod}\eqref{ker-comod0} that $[\rI,\rM]=0$ is equivalent to $\rQ$ not being in $\add\{\rF\rI\mid\rF\in\cC\}$.
By Lemma \ref{kill-simple} this is then equivalent to $\rM\square_\rC\rD \neq 0$.

As $\rM\square_\rC\rD\square_\rD\rD_\rC$ is a subcomodule of $\rM$, the assumption of $\rM$ being simple means that $\rM\square_\rC\rD\neq 0$ is equivalent to $\rM\square_\rC\rD\square_\rD\rD_\rC\cong \rM$.
\endproof

\begin{proposition}\label{fix}
Let $\rC$ be a coalgebra 1-morphism, and $\rI$ an injective $\rC$-comodule. There exists a subcoalgebra $\rD$ of $\rC$, unique up to isomorphism, which is maximal with respect to $\rI\square_\rC\rD=0$, and such that $\inj_{\uccC}(\rD)$ is equivalent to the quotient $2$-representation $\inj_{\uccC}(\rC)/\bfG_{\inj_{\uccC}(\rC)}(\rI)$. Furthermore, $\rD$ is coidempotent.
\end{proposition}

\proof
Consider the exact sequence of $2$-representations
$$0 \to \bfG_{\inj_{\uccC}(\rC)}(\rI) \to  \inj_{\uccC}(\rC) \overset{\pi}{\to} \bfK\to 0.$$
The construction in \cite[Section 3.2]{MMMZ} produces, for any full and dense morphism of $2$-representations, an embedding of a subcoalgebra, and this embedding is strict if and only if the full and dense morphism is not an equivalence.

Explicitly, in our situation, this construction defines the coalgebra $1$-morphism $\rD$ via
$$\Hom_{\bfK}(\rC, \rF\rC)\cong \Hom_{\uccC}(\rD, \rF)$$
for all $1$-morphisms $\rF$ in $\cC$ and produces an embedding $\iota:\rD\hookrightarrow\rC$.

Furthermore, $\bfK$ is equivalent to $\inj_{\uccC}(\rD)$ and the full and dense morphism of $2$-representations $\inj_{\uccC}(\rC)\twoheadrightarrow\inj_{\uccC}(\rD)$ corresponding to $\pi$ is given by $-\square_\rC\rD$ by \cite[Proposition 11]{MMMZ}. 

We assert that $\rD$ is maximal among subcoalgebras $\rB$ of $\rC$ with $\rI\square_\rC\rB=0$. Firstly, note that we indeed have $\rI\square_\rC\rD=0$ by exactness of 
\begin{equation}\label{prop18ses}0\to \bfG_{\inj_{\uccC}(\rC)}(\rI)\to \inj_{\uccC}(\rC)\xrightarrow{-\square_\rC\rD}\inj_{\uccC}(\rD) \to 0.\end{equation}
Secondly, if $\rB$ is another subcoalgebra of $\rC$, strictly containing $\rD$, then we obtain a full and dense morphism of $2$-representations $\inj_{\uccC}(\rB)\twoheadrightarrow
\inj_{\uccC}(\rD)$ that is not an equivalence.
Hence the kernel of $-\square_\rC\rB$ would be strictly contained in the kernel of $-\square_\rC\rD$.  By exactness of \eqref{prop18ses}, the latter ideal (of $\inj_{\uccC}(\rC)$) is the same as the ideal generated by $\bfG_{\inj_{\uccC}(\rC)}(\rI)$, so there must be some $\rQ\in \bfG_{\inj_{\uccC}(\rC)}(\rI)$ so that $\rQ\square_\rC\rB\neq 0$.
But $\rQ\square_\rC\rB$ is a direct summand of $\rF\rI\square_\rC\rB$, so we deduce that $\rI\square_\rC\rB\neq 0$.

It remains to show that $\rD$ is coidempotent.
We first claim that the essential image in $\comod_{\uccC}(\rD)$ of the fully faithful functor $-\square_\rD\rD_\rC$ is closed under extensions.
Indeed, assume $\rM_1,\rM_2$ are in the full subcategory given by the essential image of $-\square_\rD\rD_\rC$. In particular, for $i=1,2$,  we have $\rM_i\cong \rM_i\square_\rC\rD\square_\rD\rD_\rC$, which implies that no composition factor of $\rM_i$ is annihilated by $-\square_\rC\rD$. Consider an extension $0\to\rM_1\to \rM\to\rM_2 \to 0$ in $\comod_{\uccC}(\rC)$. Then also no composition factor of $\rM$ is annihilated by $-\square_\rC\rD$, which implies that there is no map from $\rM$ to any injective $\rQ\in \bfG_{\inj_{\uccC}(\rC)}(\rI)$.

We obtain a commutative diagram with exact rows
$$\xymatrix{
0 \ar[r]&\rM_1\square_\rC\rD\square_\rD\rD_\rC \ar^{\sim}[d]\ar[r]&\rM\square_\rC\rD\square_\rD\rD_\rC\ar@{^{(}->}[d]\ar[r]&\rM_2\square_\rC\rD\square_\rD\rD_\rC\ar^{\sim}[d]&\\
0\ar[r]&\rM_1\ar[r]&\rM\ar[r]&\rM_2\ar[r]&0
}$$
in $\comod_{\uccC}(\rC)$.

Let $q \colon \rM \to \rN$ be the cokernel map of the middle vertical map in the diagram and let 
$$\xymatrix{
\rM \ar@{^{(}->}^j[r]\ar@{->>}^{q}[d] & \rQ_1 \ar^{g}[r]\ar^{q_1}[d] &\rQ_2\ar^{q_2}[d] \\
 \rN\ar@{^{(}->}^{j'}[r]& \rQ_1' \ar^{g'}[r] &\rQ_2'
}$$
in be a lift of $q$ to an injective presentation. Since $q$ is annihilated by $-\square_\rC\rD$, the map $q_1 \square_\rC\rD$ factors over $g\square_\rC\rD$. 
By fullness of $-\square_\rC\rD$, this implies that there already is a map $h\colon \rQ_2\to \rQ_1'$ such that setting $q_1'=q_1-hg$, we have $q_1' \square_\rC\rD=0$. Note that replacing $q_1$ by $q_1'$ and $q_2$ by $q_2-g'h$ defines  another lift of $q$ to a map between injective presentations, so without loss of generality, we may assume that we already had $q_1 \square_\rC\rD=0$ by choosing $q_1$ appropriately. By exactness of \eqref{prop18ses}, this implies that 
$q_1$ factors over an object $\rF\rI$ in $ \bfG_{\inj_{\uccC}(\rC)}(\rI)$ and the first two columns of the diagram give rise to a commutative diagram
$$\xymatrix{
\rM \ar@{^{(}->}^{j}[r]\ar@{->>}^{q}[dd] & \rQ_1 \ar^{q_1}[dd] \ar[rd]& \\
&& \rF\rI\ar[ld]\\
 \rN\ar@{^{(}->}^{j'}[r]& \rQ_1'.
}$$
Now the fact that there is no nonzero map from $\rM$ to any object in $ \bfG_{\inj_{\uccC}(\rC)}(\rI)$ implies that $j'q=q_1j=0$ and by monicity of $j'$, we conclude that $q=0$. Therefore, the embedding $\rM\square_\rC\rD\square_\rD\rD_\rC \hookrightarrow \rM$ is an isomorphism, meaning that $\rM$ is in the essential image of $-\square_\rD\rD_\rC$.
This finishes the proof for the claim that the essential image of $-\square_\rD\rD_\rC$ is extension-closed.

By Lemma \ref{lemclosed}, we already know that $-\square_\rD\rD_\rC$ embeds $\comod_{\uccC}(\rD)$ as a full subcategory of $\comod_{\uccC}(\rC)$ that is closed under subobjects and quotients.
So the essential image of $-\square_\rD\rD_\rC$ being also closed under extensions implies that it is a Serre subcategory.
Now the statement that $\rD$ is coidempotent follows from Proposition \ref{NT-corresp}, whereas the uniqueness of $\rD$ follows from Proposition \ref{cat-coalg-bij}.\endproof

\subsection{The main result.}
\begin{theorem}\label{thm-main}
\begin{enumerate}[$($i$)$]
\item\label{thm1.1}
Let $\rC$ be a coalgebra $1$-morphism in $\underline{\cC}$ and suppose $\rD$ is a coidempotent subcoalgebra  of $\rC$. Set $\rJ:=\rC/\rD$. Let $\rI$ be the injective hull of $\rJ$ inside $\comod_{\underline{\ccC}}(\rC)$ and set $\rA=[\rI,\rI]$.Then we have a short exact sequence of $2$-representations
$$0 \longrightarrow \inj_{\uccC}(\rA) \xrightarrow{-\square_\rA\rI} \inj_{\uccC}(\rC)\xrightarrow{-\square_\rC\rD}\inj_{\uccC}(\rD) \longrightarrow 0.$$

\item\label{thm1.2} Suppose $$0\longrightarrow\bfN\longrightarrow\bfM\longrightarrow\bfK\longrightarrow 0$$ is a short exact sequence of $2$-representations. 
Then, choosing a coalgebra $1$-morphism $\rC$ with $\bfM\cong \inj_{\uccC}(\rC)$, there exists a subcoalgebra $\rD$ of $\rC$, unique up to isomorphism, which is maximal with respect to $X\square_\rC\rD=0$ for all $X\in \coprod_{\ti\in\ccC} \bfN(\ti)$, and such that $\inj_{\uccC}(\rD)$ is equivalent to the quotient $2$-representation $\bfK$. Furthermore, $\rD$ is coidempotent. 
\end{enumerate}
\end{theorem}

\proof
\eqref{thm1.1} The assumption is precisely that of Proposition \ref{recollcondition}, so we obtain a recollement of the form \eqref{recollement}.  Now the claim follows from Lemma \ref{recolllem}.

\eqref{thm1.2} 
Choose $\rI$ such that, under the equivalence $\bfM\cong \inj_{\uccC}(\rC)$, the $2$-subrepresentation $\bfN$ corresponds to $\bfG_{\inj_{\uccC}(\rC)}(\rI)$. Then, for a subcoalgebra $\rD$ of $\rC$, we have $X\square_\rC\rD=0$ for all $X\in \coprod_{\ti\in\ccC} \bfN(\ti)$ if and only if $\rI\square_\rC\rD=0$.
The claim now follows from Proposition \ref{fix}.
\endproof

\section{Examples}\label{sec:eg}

\subsection{Projective functors over dual numbers}
Let $R=\Bbbk[x]/(x^2)$ be the ring of dual numbers.
Consider the $2$-category $\cC_R$ of projective functors on $R\lmod$, see e.g. \cite[Example 2]{MM1}.
More precisely, this is the $2$-category with one object $\ti$, which we identify with a small category $\mathcal{R}$ equivalent to $R\lmod$, and $\cC_R(\ti,\ti)$ is the full subcategory of all endofunctors of $\mathcal{R}$ given by functors isomorphic to tensoring over $R$ with an $R$-$R$-bimodule in $\add(R\oplus R\otimes_{\Bbbk}R)$.

The $2$-category $\cC_R$ has two indecomposable $1$-morphisms $\one$ and $\rF$ corresponding to the identity functor and to tensoring with $R\otimes R$, respectively.

Since the principal $2$-representation $\bfP=\cC_R(\ti,-)$ is generated by $\one$, we have a coalgebra $1$-morphism $\rC_\bfP$ corresponding to $\bfP$ given by  $[\one,\one]$.
Using the fact that the underlying category of $\underline{\bfP}$ is precisely $\ucC_R$, one can see that $[\one,\one]\cong\one$ as an object in $\ucC_R(\ti,\ti)$ and comultiplication and counit are both the identity map.  Note that this argument applies for any principal $2$-representation of a finitary $2$-category.

There are two simple transitive $2$-representations, denoted by $\bfC_{\L}, \bfC_{\one}$, up to equivalence (see \cite{MM5} for details). Here
$\bfC_{\one}$ is the trivial $2$-representation, whose underlying category is equivalent to $\Bbbk\lmod$, where $\rF$ acts by annihilating everything. On the other hand, $\bfC_{\L}$ is the natural $2$-representation, whose underlying category is equivalent to $R\proj$, where $\rF$ acts as $R\otimes_{\Bbbk}R\otimes_R-$.

It follows from \cite[Theorem 22]{MMMT} that the as an object of  $\ucC_R(\ti,\ti)$, the coalgebra $1$-morphism $\rC_\L:=[R,R]$ is isomorphic to $\rF$.
For $\bfC_{\one}$, the corresponding coalgebra $1$-morphism $\rC_\one:=[\Bbbk,\Bbbk]$ can be calculated via the defining adjunction isomorphisms $\Hom_{\uccC_R}(\rC_\one,\rG)\cong \Hom_{\underline{\bfC}_\one(\ti)}(\Bbbk,\rG\Bbbk)$ for all indecomposable $1$-morphisms $\rG$, which yield that it is is isomorphic to the simple socle $L_\one$ of $\one$ in $\ucC_R(\ti,\ti)$.
In fact, $\bfC_{\one}$ is a quotient $2$-representation of $\bfP$, so $\rC_\one$ is a subcoalgebra of $\rC_\bfP$, which implies that the counit and comultiplication maps 
defining $\rC_\one$ are both the identity map on $L_\one$.

There is a short exact sequence of finitary $2$-representations
\[
0 \to \bfC_{\L}\boxtimes \A \to \bfP \to \bfC_{\one} \to 0,
\]
where $\bfC_{\L}\boxtimes \A$ is the inflation of $\bfC_{\L}$ by $\A:=R\proj$ (see \cite{MM6} for details about inflations).  Computing coalgebra $1$-morphism corresponding to $\bfC_{\L}\boxtimes \A$ via the defining adjunction isomorphism $\Hom_{\uccC}([R\otimes R,R\otimes R],\rF)\cong \Hom_{\underline{\bfC_{\L}\boxtimes \A}}(R\otimes R,\rF(R\otimes R))$ shows that its underlying object in $\ucC_R(\ti,\ti)$ is isomorphic to $\rF\oplus \rF$.
Thus, the coalgebra $1$-morphisms $\rD, \rC$ in Theorem \ref{thm-main} corresponding to the above short exact sequence are $\rC_\one$, $\rC_\bfP$.
Since the quotient of $\one$ by $L_\one$ has simple socle $L_\rF$, the coalgebra $1$-morphism is $\rA=[\rF,\rF]$ and it has underlying object  in $\ucC_R(\ti,\ti)$ given by $\rF\oplus \rF$.

Let us look at another finitary $2$-representation $\bfM$, whose underlying category is equivalent to $\Lambda\proj$, where $\Lambda:=\End_R(R\oplus \Bbbk)$ (the action of $\cC_R$ is naturally induced by that on the natural $2$-representation $\bfC_\L$).
There is a short exact sequence of 
\[
0 \to \bfC_{\L} \to \bfM \to \bfC_{\one} \to 0.
\]
We already explained that the coalgebra $1$-morphisms $\rA,\rD$ corresponding to the first and last term, respectively, are $\rF$ and $L_\one$.
It is possible to calculate the object in $\ucC_R(\ti,\ti)$ underlying the coalgebra $1$-morphism $\rC_\bfM$ corresponding to $\bfM$ as follows.

First note that there is a quotient morphism $\bfP\to \bfM$, so we can take $\rC_\bfM$ is to be the subcoalgebra of $\rC_\bfP$ given by $[M,M]$ with $M$ being the underlying object of $\rC_\bfM$ in $\ucC=\bfP(\ti)$.
Note that the internal Hom used here is taken in $\bfP$ (instead of $\bfM$).
The underlying object of $\rC_\bfP$ is the injective object $\one$ of $\ucC_R(\ti, \ti)$, and $\one$ turns out to be uniserial with four composition factor, $L_\rF, L_\one, L_\rF, L_\one$ from top to socle, where $L_\rF$ is the simple socle of $\rF$.
By the above short exact sequence, $\rC_\one$ is coidempotent subcoalgebra of $\rC_\bfM$ and $\rC_\one\ncong \rC_\bfM$, so the underlying object is not $L_\one$. 
Since $\rC_\bfP \ncong \rC_\bfM$, the underlying object $M$ of $\rC_\bfM$ can only be either the length 2 or the length 3 subobject of $\one$.
These two objects can be distinguished by the dimension of the Hom-space of maps to $\one$ in $\ucC$ - they are of dimension 1 and 2 respectively.
By construction, $M\cong [M,M]=\rC_\bfM$ as objects in $\bfP(\ti)$ (c.f. Proof of Lemma \ref{realise-subcat}) and $\Hom_{\uccC}(\rC_\bfM,\one)\cong \Hom_{\underline{\bfP}}(M, M)$.
Note that $\Hom_{\underline{\bfP}}(M,M)\cong \Hom_{\underline{\bfM}}(M,M)$ as $\underline{\bfM}(\ti)\to\underline{\bfP}(\ti)$ is a fully faithful embedding.

We claim that $ \Hom_{\underline{\bfM}}(M,M)\cong \Bbbk$; in which case, we can conclude that $M$ is the subobject of $\one$ of length 2.
Indeed, under the equivalence between the underlying category of $\bfM$ and $\Lambda\proj$, $M$ corresponds to the indecomposable projective $\Lambda$-module $P$ such that $\rF(P)\notin \add(P)$.  By the construction of 2-representation structure on $\Lambda\proj$, $\rF(P')\in\add (P')$ for an indecomposable projective $\Lambda$-module $P'$ if and only if $P'\cong\Hom_R(R\oplus \Bbbk, R)$.  So we have $P\cong \Hom_R(R\oplus \Bbbk, \Bbbk)$, which is a uniserial module with a 1-dimensional endomorphism ring; this means that $\Hom_{\underline{\bfM}}(M,M)=\Bbbk$, as claimed.

Note that the comultiplication and counit maps defining $\rC_\bfM$ are both identity, since $\rC_\bfM$ is a subcoalgebra of $\rC_\bfP$.


Let us summarise, for clarity, the object underlying each coalgebra $1$-morphism corresponding to the $2$-representations mentioned, in the table below.
\[
\begin{array}{ccccc}
\bfP & \bfC_\one & \bfC_\L & \bfC_\L\boxtimes \A & \bfM \\
\hline 
\one & L_\one & \rF & \rF\oplus \rF & \begin{array}{c} L_\rF \\ L_\one \end{array} 
\end{array}
\]

\subsection{Triangular coalgebras}

In this section, we provide a more general class of examples. Let $\cC$ be a weakly fiat $2$-category, $\rA$ and $\rD$ coalgebra $1$-morphisms in $\cC$ with comultiplications and counits given by $\mu_\rA,\epsilon_\rA$ and $\mu_\rD,\epsilon_\rD$, respectively, and ${}_\rA\rM_\rD$ a $\rA$-$\rD$-bicomodule with left and right coactions $\lambda$ respectively $\rho$. We define a coalgebra structure $\mu_\rC,\epsilon_\rC$ on $\rC:=\rA\oplus\rD\oplus\rM$ by specifying that
\begin{itemize}
\item the restriction of $\mu_\rC$ to $\rX$ is $\mu_\rX$ for $\rX\in\{\rC,\rD\}$,
\item the restriction of $\epsilon_\rC$ to $\rX$ is $\epsilon_\rX$ for $\rX\in\{\rC,\rD\}$,
\item the restriction of $\mu_\rC$ to $\rM$ is $\left(\begin{array}{c}\lambda\\ \mu\end{array}\right)\colon \rM\to \rA\rM\oplus \rM\rD$, and
\item the restriction of $\epsilon_\rC$ to $\rM$ is zero.
\end{itemize}
It is straightforward to check that this indeed defines a coalgebra structure on $\rC$.

We claim that $\rD$ is a coidempotent subcoalgebra. It is a subcoalgebra by definition, so we need to check that $\mu_\rC^{-1}(\rC\rD+\rD\rC) = \rD$. Now we have $$\rC\rD+\rD\rC\cong \rA\rD\oplus \rD\rD\oplus \rM\rD + \rD\rA \oplus\rD\rM\oplus \rD\rD.$$

As none of $\rA\rD, \rD \rA, \rD\rM$ is in the range of $\mu_\rC$, we have $\mu_\rC^{-1}(\rC\rD+\rD\rC)=\mu_\rC^{-1}(\rD\rD\oplus \rM\rD)$.  
Since $\mu_\rC$ sends $\rM$ to $\rA\rM\oplus \rM\rD$, but only $\rM\rD$ is a direct summand of $\rC\rD+\rD\rC$, we get that 
$$\mu_{\rC}^{-1}(\rD\rD\oplus \rM\rD)=\mu_\rD^{-1}(\rD\rD)\oplus \left(\lambda^{-1}(0)\cap \rho^{-1}(\rM\rD)\right).$$
It follows from the construction that $\mu_\rD^{-1}(\rD\rD)=\rD$, whereas $\lambda^{-1}(0)=0$ due to $\lambda$ being mono, so we obtain $\mu_{\rC}^{-1}(\rC\rD+\rD\rC)=\rD$, i.e. $\rD$ is coidempotent as claimed.

%
%
%
 

\end{document}